\documentclass[11pt, leqno]{article} 
\usepackage[a4paper, top=2.5cm, bottom=2.5cm, left=2.5cm, right=2.5cm]{geometry}
\usepackage[english]{babel}
\usepackage{enumitem}
\setlist[itemize]{label=$\bullet$}

\usepackage{xcolor}

\usepackage{amsmath}  % Advanced math typesetting
\usepackage{amsthm}   % Theorem environments and proof support
\usepackage{amssymb}  % Mathematical symbols

\newtheorem{theorem}{Theorem}[section]
\newtheorem{lemma}[theorem]{Lemma}
\newtheorem{proposition}[theorem]{Proposition}
\newtheorem{remark}[theorem]{Remark}
\newtheorem{counterexample}[theorem]{Counterexample}
\newtheorem{fact}[theorem]{Fact}
\newtheorem{corollary}[theorem]{Corollary}
\numberwithin{equation}{section}
\usepackage{dsfont}
\usepackage{hyperref}
\hypersetup{
    colorlinks=true,
    linkcolor=blue,
    citecolor=cyan,
    urlcolor=teal,
}
\theoremstyle{definition}

\title{\textbf{Dimension dependence and dimension-free $\ell^2$~estimates for variation seminorms of spherical means on~the~hypercube}}
\author{
    \textbf{Rafa{\l} {\L}y\.zwa}%
    \thanks{Mathematical Institute, University of Wroc{\l}aw, pl. Grunwaldzki 2, 50--384 Wroc{\l}aw, Poland.\\
    Email: \texttt{rafal.lyzwa@math.uni.wroc.pl}\\
    Research supported by the National Science Centre, Poland (Narodowe Centrum Nauki), Grant 2022/46/E/ST1/00036.
    }
}
\date{}

\begin{document}

\maketitle

\begin{abstract}
    We prove that the $\ell^2 \to \ell^2$ norm of the $r$-variation seminorm of spherical means on the hypercube admits no dimension-free bounds for any $r \in [1, \infty)$ when the variation is taken over all possible radii. Furthermore, we establish that if the radii are restricted to a fixed parity, dimension-free estimates hold for all $r \in (2, \infty)$.
\end{abstract}

% Keywords and MSC Codes
\vspace{1em}


%\tableofcontents

\section{Introduction}

\subsection{Motivation}

Let $G \subseteq \mathbb{R}^n$ be a centrally symmetric convex body (i.e., a closed, bounded, and centrally symmetric convex set with a non-empty interior). For $f \in L^1_{\text{loc}}(\mathbb{R}^n)$ and $t > 0$, we define the integral Hardy--Littlewood averaging operator associated with $G$ as
\begin{equation*}
    \mathbf M_t^G f(x) := \frac{1}{|tG|} \int_{tG} f(x - y) \, dy.
\end{equation*}
For $p \in (1, \infty]$, let $C_p(n, G)$ be the best constant in the maximal inequality
\begin{equation*}
    \left\| \sup_{t>0} |\mathbf M_t^G f| \right\|_{L^p(\mathbb{R}^n)} \leqslant C_p(n, G) \big\|f\big\|_{L^p(\mathbb{R}^n)}.
\end{equation*}
It is well known that $C_p(n, G) < \infty$ for all $p \in (1, \infty]$. A fundamental problem in this context is the existence of a dimension-independent estimate of $C_p(n, G)$. In the case of Euclidean balls ($G = B^2$), Stein \cite{Stein82} proved that $C_p(n, B^2)$ is bounded independently of $n$ for all ${p \in (1, \infty]}$. For general centrally symmetric convex bodies, the study of high-dimensional bounds in geometry and analysis was initiated by Bourgain in his seminal work \cite{Bourgain}. There, he established the existence of a dimension-free estimate of $C_2(n,G)$ (the case $p=2$), a~result that was subsequently extended to $p \in \big(\frac 32, \infty\big]$ by Bourgain \cite{Bourgain-2} and Carbery \cite{Carbery}. Additionally, if the supremum is restricted to the dyadic set $t \in \{2^m : m\in\{0,1,2,...\}\}$, the constant $C_p(n, G)$ has a bound independent of the body $G$ for the full range $p \in (1, \infty]$.

It is a major open conjecture that the maximal inequality holds for all $p \in (1, \infty]$ and for all centrally symmetric convex bodies $G \subseteq \mathbb{R}^n$, with the constant $C_p(n, G)$ bounded independently of the dimension $n$. This conjecture is well-founded, having already been verified for substantial classes of convex bodies. In particular, for the family of $\ell^q$-balls $G = B^q$, dimension-free estimates covering the full range $p \in (1, \infty]$ have been successfully established. Specifically, Müller~\cite{Muller} resolved the case for $q \in [1, \infty)$, while Bourgain~\cite{Bourgain-3} settled the case of the cube ($q = \infty$), with both results yielding constants that depend exclusively on $p$ and $q$.

More recently, these dimension-free results have been significantly strengthened to the setting of variational inequalities. Bourgain, Mirek, Stein, and Wróbel \cite{BMSW-GAFA} established variational estimates that provide quantitative control over the fluctuations of the averages, independent of the dimension.

Dimension-independence in maximal and variational inequalities has also been explored in various discrete settings. Bourgain, Mirek, Stein, and Wróbel \cite{BMSW} established dimension-free $\ell^p$~estimates for the maximal operator associated with discrete means over cubes in $\mathbb{Z}^n$ for every $p\in \big(\frac 32,\infty\big]$. In the same setting, they proved dimension-free bounds for the corresponding \linebreak $r$-variation seminorm for all $p\in \big(\frac 32,4\big)$ and $r\in (2,\infty)$. In the context of the hypercube $\{0,1\}^n$, Harrow, Kolla, and Schulman \cite{HKS} proved dimension-free $\ell^2$ estimates for the maximal operator of spherical means on the hypercube, a result that Krause \cite{Krause} subsequently extended to all $\ell^p$ for $p>1$.

A key tool used to obtain the results of this paper is the family of Krawtchouk polynomials. In the context of dimension-free bounds, these polynomials have recently been employed by Harrow, Kolla, and Schulman \cite{HKS}, Kosz, Niksiński, and Wróbel \cite{KNW}, Krause \cite{Krause}, Mirek, Szarek, and Wróbel \cite{MSW}, Niksiński \cite{Niksinski-1, Niksinski-2}, and Niksiński and Wróbel \cite{NW}.

\subsection{Statement of the results}

This paper focuses on dimension-free $\ell^2$ estimates for the variation seminorms of spherical means on the hypercube. Before stating the main results, we introduce the necessary notation and definitions.
\begin{itemize}
    \item Let $n\in\{ 1,2,3,...\}$ and let $\mathbb I^n:=\{0,1\}^n$ be the $n$-dimensional hypercube (also called the Hamming cube) equipped with the metric
    \begin{equation*}
        d(x,y):=|\{ j:x(j)\neq y(j)\}|\qquad\text{for }x,y\in\mathbb I^n.
    \end{equation*}
    \item For $x\in\mathbb I^n$, we define its length as $|x|:=\sum_jx(j)$.
    \item For every $k\in\{0,1,...,n\}$ we define the spherical mean $S_k$ by
    \begin{equation*}
        S_kf(x):=\frac 1{\binom nk} \sum_{\begin{smallmatrix}
            y\in\mathbb I^n\\
            d(x,y)=k
        \end{smallmatrix}} f(y)\qquad\text{for }f:\mathbb I^n\to\mathbb C\text{ and }x\in\mathbb I^n.
    \end{equation*}
    \item Let $\widetilde{\chi_y}(x):=\frac 1{\sqrt{2^n}}(-1)^{x\cdot y}$ for $x,y\in\mathbb I^n$, where $\cdot$ denotes the standard inner product on $\mathbb R^n$. The system $(\widetilde{\chi_y}:y\in\mathbb I^n)$ forms a linear (and orthonormal) basis for $\ell^2(\mathbb I^n)$. Moreover, for every $f:\mathbb I^n\to\mathbb C$, we have
    \begin{equation}\label{eq41}
        f=\sum_{y\in\mathbb I^n}\widehat f(y)\widetilde{\chi_y}.
    \end{equation}
    \item Let $r\in [1,\infty)$. For any $Z\subseteq\{0,1,...,n\},\;\; f:\mathbb I^n\to\mathbb C$, and $x\in\mathbb I^n$, we define the $r$-variation seminorm of the mapping $Z\ni k\mapsto S_kf(x)$ by
    \begin{equation*}
        V_r(S_kf(x):k\in Z):=\max_{\begin{smallmatrix}
            k_0<k_1<...<k_J\\
            k_j\in Z
        \end{smallmatrix}} \Bigg(\sum_{j=1}^J |S_{k_{j-1}}f(x)-S_{k_j}f(x)|^r\Bigg)^{\frac 1r}.
    \end{equation*}
\end{itemize}

We now present the main results of this paper. It turns out that the quantity
\begin{equation}\label{eq39}
    \sup_{f\in\mathbb C^{\mathbb I^n}\setminus\{ 0\}}\frac {\big\| V_r(S_kf:k\in\{0,1,...,n\})\big\|_{\ell^2(\mathbb I^n)}}{\big\|f\big\|_{\ell^2(\mathbb I^n)}}
\end{equation}
admits no dimension-free bounds for any $r\geqslant 1$. Equivalently, we establish the following theorem:

\begin{theorem}
    For every $r\in [1,\infty)$, we have
    \begin{equation*}
        \sup_n \sup_{f\in\mathbb C^{\mathbb I^n}\setminus\{ 0\}}\frac {\big\| V_r(S_kf:k\in\{0,1,...,n\})\big\|_{\ell^2(\mathbb I^n)}}{\big\|f\big\|_{\ell^2(\mathbb I^n)}} =\infty.
    \end{equation*}
\end{theorem}

This theorem is a direct consequence of Counterexample \ref{thm:kontrprz_1}. This counterexample essentially demonstrates that for every $r\geqslant 1$ and every dimension $n$, we have
\begin{equation*}
    \big\|V_r(S_k\widetilde{\chi_{\mathbf 1_n}}:k\in\{0,1,...,n\})\big\|_{\ell^2(\mathbb I^n)}\geqslant 2n^{\frac 1r} \big\|\widetilde{\chi_{\mathbf 1_n}}\big\|_{\ell^2(\mathbb I^n)},
\end{equation*}
where $\mathbf 1_n$ denotes the element of $\mathbb I^n$ consisting entirely of ones.

Let us note that $\mathbf 1_n$ is the element of $\mathbb I^n$ with the maximal possible length. A natural question arises: could the failure of dimension-free estimates be attributed to the influence of functions $\widetilde{\chi_y}$ for large $|y|$? One might expect that restricting our attention to functions $f$ with $\widehat f(y)=0$ for large $|y|$ (see \eqref{eq41}) would yield dimension-free bounds. In other words, one might ask whether \eqref{eq40} holds for certain specific subsets $E_n\subseteq\mathbb I^n$ (roughly of the form $E_n=\{y\in\mathbb I^n:|y|\text{ is large}\}$). It turns out, however, that merely excluding the influence of $\widetilde{\chi_{\mathbf 1_n}}$, or even the functions $\widetilde{\chi_y}$ for $|y|\geqslant n-b_n$ (where $(b_n)_{n=1}^\infty$ is a sequence of positive real numbers such that $\big(\frac n{b_n}\big)_{n=1}^\infty$ is unbounded, e.g., $b_n:=n^\alpha$ with $\alpha\in (0,1)$), is insufficient. Specifically, by setting $E_n:=\{y\in\mathbb I^n:|y|\geqslant n-b_n\}$, we show that \eqref{eq40} remains false for all $r\geqslant 1$. This is formalized in Corollary \ref{thm:wn_z_kontrprz_2}. Conversely, positive results can be obtained under a more restrictive truncation:

\begin{proposition}
    Let $E_n:=\{y\in\mathbb I^n:|y|>\frac n2\}$ for $n\in\{ 1,2,3,...\}$. Then for every $r\in (2,\infty)$, we have
    \begin{equation}\label{eq40}
        \sup_n \sup_{\begin{smallmatrix}
            f\in\mathbb C^{\mathbb I^n}\setminus\{ 0\}\\
            \widehat{\,f\,}\upharpoonright_{E_n}\equiv\, 0
        \end{smallmatrix}}\frac {\big\| V_r(S_kf:k\in\{0,1,...,n\})\big\|_{\ell^2(\mathbb I^n)}}{\big\|f\big\|_{\ell^2(\mathbb I^n)}} <\infty.
    \end{equation}
\end{proposition}

\noindent This proposition will be restated later as Proposition \ref{thm:stw_1}.

The failure of dimension-free bounds for \eqref{eq39} is closely tied to the fact that ${\big|(-1)^{k_1}-(-1)^{k_2}\big|}$ $=2$ for indices $k_1,k_2\in\{0,1,...,n\}$ of different parities. We demonstrate that restricting the variation to radii of a fixed parity restores the dimension-free estimates. More precisely, we establish the following:

\begin{theorem}
    For every $r\in (2,\infty)$ and $q\in\{0,1\}$, we have
    \begin{equation*}
        \sup_n \sup_{f\in\mathbb C^{\mathbb I^n}\setminus\{ 0\}}\frac {\big\| V_r(S_kf:k\in (2\mathbb Z+q)\cap\{0,1,...,n\})\big\|_{\ell^2(\mathbb I^n)}}{\big\|f\big\|_{\ell^2(\mathbb I^n)}} <\infty.
    \end{equation*}
\end{theorem}

\noindent This theorem constitutes the main result of the paper and will be restated later as Theorem~\ref{thm:glowne_tw}.

The proofs presented herein are partially based on the techniques developed in \cite{BMSW}.

\section{Preliminaries and notation}

\begin{itemize}
\item
Throughout the paper the number $n\in\{1,2,3,...\}$ (which can be thought to be large) will denote the dimension.
\item
For two integers $k,l$ satisfying $k\leqslant l$ we use the notation
\begin{equation*}
    \{k,...,l\}:=\{m\in\mathbb Z:k\leqslant m\leqslant l\}.
\end{equation*}
\item
For two real-valued expressions $\varphi$ and $\psi$ we write
\begin{equation*}
    \varphi\lesssim\psi
\end{equation*}
meaning that
\begin{equation*}
    \varphi\leqslant C\psi
\end{equation*}
for some universal constant $C>0$.
\item
Similarly, for a parameter $a$ we write
\begin{equation*}
    \varphi\lesssim_a\psi
\end{equation*}
meaning that
\begin{equation*}
    \varphi\leqslant C_a\psi
\end{equation*}
for some constant $C_a>0$ dependent only on the parameter $a$.
\end{itemize}

\subsection{The Hamming cube, the spherical means, and the corresponding Hilbert space}

\begin{itemize}
\item
Let $\mathbb I$ denote the group $\{0,1\}$ with addition modulo 2. The Hamming cube is the group $\mathbb I^n$ with the product action, which we call addition and denote by $\oplus$. Let us note that in $\mathbb I^n$ subtraction ($\ominus$) is the same as addition.
\item
For a subset $A\subseteq\mathbb I^n$ we denote the characteristic function of $A$ by $\mathds 1_A$.
\item
We define the convolution $f*g$ of functions $f,g:\mathbb I^n\to\mathbb C$ by
\begin{equation*}
    f*g\: (x):=\sum_{y\in\mathbb I^n}f(x\ominus y)g(y)=\sum_{y\in\mathbb I^n}f(x\oplus y)g(y)\qquad\text{for }x\in\mathbb I^n.
\end{equation*}
\item
For $x\in\mathbb I^n$ we write
\begin{equation*}
    x=\Big(x(1),...,x(n)\Big)
\end{equation*}
and define the length $|x|$ of $x$ by
\begin{equation*}
    |x|:=\sum_{j=1}^nx(j),
\end{equation*}
where $\sum$ denotes the usual sum in $\mathbb C$.
\item
For $k\in\{0,...,n\}$ we define the sphere $R_k$ of radius $k$ by
\begin{equation*}
    R_k:=\{x\in\mathbb I^n:|x|=k\}
\end{equation*}
and the spherical mean $S_k$ by
\begin{equation*}
    S_kf:=f*\tfrac 1{|R_k|}\mathds 1_{R_k}\qquad\text{for }f:\mathbb I^n\to\mathbb C.
\end{equation*}
\item
On the linear space $\mathbb C^{\mathbb I^n}$ we define the inner product and the second norm:
\begin{equation*}
    \begin{split}
        &\langle f,g\rangle:=\sum_{x\in\mathbb I^n}f(x)\overline{g(x)}\qquad\text{for } f,g:\mathbb I^n\to\mathbb C,\\
        &\|f\|_2:=\sqrt{\langle f,f\rangle}\qquad\qquad\:\,\text{for } f:\mathbb I^n\to\mathbb C.
    \end{split}
\end{equation*}
The Hilbert space $\ell^2(\mathbb I^n)$ is the linear space $\mathbb C^{\mathbb I^n}$ equipped with the inner product $\langle\cdot,\cdot\rangle$.
\item
We also define the $p$-th norms on $\mathbb C^{\mathbb I^n}$:
\begin{equation*}
    \begin{split}
        &\|f\|_p:=\left(\sum_{x\in\mathbb I^n}|f(x)|^p\right)^{\frac 1p} \qquad\text{for }f:\mathbb I^n\to\mathbb C\text{ and }p\in[1,\infty),\\
        &\|f\|_\infty:=\max_{x\in\mathbb I^n}|f(x)| \qquad\quad\;\;\;\;\text{for }f:\mathbb I^n\to\mathbb C.
    \end{split}
\end{equation*}
\end{itemize}

We have the following lemma relating to the spherical means $S_k$. Its proof is standard so we omit it.

\begin{lemma}\label{thm:Sk_kontrakcja}
    For any $p\in[1,\infty],\;\; k\in\{0,...,n\}$, and $f:\mathbb I^n\to\mathbb C$ we have
    \begin{equation*}
        \|S_kf\|_p\leqslant\|f\|_p.
    \end{equation*}
\end{lemma}

% \begin{proof}
%     Let us denote $F_y(x):=f(x\ominus y)$ for $x,y\in\mathbb I^n$. We have
%     \begin{equation*}
%         \|S_kf\|_p=\left\| \frac 1{|R_k|}\sum_{y\in R_k}F_y \right\|_p\leqslant\frac 1{|R_k|}\sum_{y\in R_k}\|F_y\|_p=\frac 1{|R_k|}\sum_{y\in R_k}\|f\|_p=\|f\|_p.
%     \end{equation*}
% \end{proof}

\subsection{Orthogonal systems and the Fourier transform}
\begin{itemize}
\item
Let $\mathbb T$ be the group $\{z\in\mathbb C:|z|=1\}$ with multiplication and let
\begin{equation*}
    \widehat{\mathbb I^n}:=\left\{ f\! :\!\mathbb I^n\!\to\!\mathbb T\; :\;f\text{ is a group homomorphism} \right\}
\end{equation*}
be the dual group of $\mathbb I^n$.
\item
For $\chi\in\widehat{\mathbb I^n}$ we define its normalization $\widetilde\chi$ by
\begin{equation*}
    \widetilde\chi:=\frac 1{\|\chi\|_2}\chi=\frac 1{\sqrt{2^n}}\chi.
\end{equation*}
\item
We have the group isomorphism $\mathbb I^n\ni y\mapsto\chi_y\in\widehat{\mathbb I^n}$ defined by
\begin{equation*}
    \chi_y(x):=(-1)^{x\cdot y}\qquad\text{for }x,y\in\mathbb I^n,
\end{equation*}
where $x\cdot y=\sum_{j=1}^nx(j)y(j)\;$ ($\cdot$ is the standard inner product on $\mathbb R^n$).\\
The system $(\chi_y:y\in\mathbb I^n)$ is a linear basis and an orthogonal system in $\ell^2(\mathbb I^n)$. The system $(\widetilde{\chi_y}:y\in\mathbb I^n)$ is an orthonormal basis of $\ell^2(\mathbb I^n)$.
\item
For $f:\mathbb I^n\to\mathbb C$ we define its Fourier transform $\widehat f$ by
\begin{equation*}
    \widehat f(y):=\frac 1{\sqrt{2^n}}\sum_{x\in\mathbb I^n}f(x)\overline{\chi_y(x)}\qquad\text{for }y\in\mathbb I^n.
\end{equation*}
Let us note that $\widehat f(y)=\langle f,\widetilde{\chi_y}\rangle$ so we have the equality
\begin{equation}\label{eq01}
    f=\sum_{y\in\mathbb I^n}\widehat f(y)\widetilde{\chi_y}
\end{equation}
and the Plancherel's equality
\begin{equation}\label{eq02}
    \|f\|_2^2=\sum_{y\in\mathbb I^n}|\widehat f(y)|^2.
\end{equation}
\end{itemize}

\subsection{The Krawtchouk polynomials and the noise semigroup}
It is easy to show that for $k\in\{0,...,n\}$ and $y\in\mathbb I^n$ we have
\begin{equation*}
    S_k\chi_y=\underbrace{\left( \frac 1{|R_k|}\sum_{w\in R_k}(-1)^{w\cdot y} \right)}_{\qquad\:\;\;\;\, =:\:\mathfrak m_k(y)} \chi_y.
\end{equation*}
A simple calculation gives us that $\mathfrak m_k(y)=\kappa^{(n)}_k(|y|)$, where
\begin{equation}\label{eq04}
    \kappa^{(n)}_k(x):=\frac 1{\binom nk}\sum_{j=0}^k(-1)^j\binom xj\binom{n-x}{k-j}
\end{equation}
is the normalized $k$-th Krawtchouk polynomial. In summary, we have the following formula:
\begin{equation}\label{eq47}
    S_k\chi_y=\kappa^{(n)}_k(|y|)\chi_y\qquad\text{for }k\in\{0,...,n\}\text{ and }y\in\mathbb I^n.
\end{equation}
Now we state some straightforward properties of Krawtchouk polynomials:

\begin{fact}\label{thm:podst_wlasn_krawtch}
    For $k,x\in\{0,...,n\}$ we have
    \begin{enumerate}
        \item $|\kappa^{(n)}_k(x)|\leqslant 1$,
        \item $\kappa^{(n)}_k(0)=1$,
        \item $\kappa^{(n)}_k(x)=\kappa^{(n)}_x(k)$,
        \item $\kappa^{(n)}_k(x)=(-1)^k\kappa^{(n)}_k(n-x)$.
    \end{enumerate}
\end{fact}

The next two lemmas are contained in \cite{HKS} (see \cite[Lemma 2.2]{HKS} for Lemma \ref{thm:kappa_jak_e} and \cite[computations on pages 65--66]{HKS} for Lemma \ref{thm:roznica_krawtch}).

\begin{lemma}\label{thm:kappa_jak_e}
    There exists a constant $c>0$ independent of anything such that for $k,x\in\{0,...,\lfloor\frac n2\rfloor\}$ we have
    \begin{equation*}
        |\kappa^{(n)}_k(x)|\leqslant e^{-c\frac{kx}n}.
    \end{equation*}
\end{lemma}

\begin{lemma}\label{thm:roznica_krawtch}
    For $n\in\{2,3,4,...\}$ and $k,x\in\{1,...,n\}$ we have the formula
    \begin{equation*}
        \kappa^{(n)}_k(x)-\kappa^{(n)}_k(x-1)= -2\frac kn \kappa^{(n-1)}_{k-1}(x-1).
    \end{equation*}
\end{lemma}

For $t\in [0,\infty)$ we define the linear operator $N_t$ on the space $\mathbb C^{\mathbb I^n}$ by
\begin{equation}\label{eq48}
    N_t\chi_y:=e^{-t|y|}\chi_y\qquad\text{for }y\in\mathbb I^n.
\end{equation}
It is easy to note that the operators $N_t,\:\: t\in [0,\infty)$, form a uniformly continuous semigroup on $\ell^2(\mathbb I^n)$, which we call the noise semigroup.

The next lemma provides another formula for the operators $N_t$.

\begin{lemma}\label{thm:inny_wzor_na_Nt}
    For $t\in [0,\infty)$ let $u_t:=\frac{1-e^{-t}}2$ and
    \begin{equation}\label{eq03}
        M_t:=\sum_{k=0}^n\binom nku_t^k(1-u_t)^{n-k} S_k.
    \end{equation}
    Then $M_t=N_t$.
\end{lemma}

In \cite{HKS} the noise semigroup is defined by the formula \eqref{eq03}.

\begin{proof}[Proof of Lemma \ref{thm:inny_wzor_na_Nt}]
    We will show that the operators $M_t$ and $N_t$ agree on the basis $(\chi_y:y\in\mathbb I^n)$. We have
    \begin{equation*}
        M_t\chi_y=\underbrace{\left( \sum_{k=0}^n\binom nku_t^k(1-u_t)^{n-k}\kappa^{(n)}_k(|y|) \right)}_{\qquad\quad\; =:\; \varphi_t(|y|)} \chi_y.
    \end{equation*}
    It suffices to show that for $x\in\{0,...,n\}$ we get $\varphi_t(x)=e^{-tx}$.\\
    Firstly let us note that for $j\in\mathbb Z$, if $j>x$ then $\binom xj=0$ while if $j<x+k-n$ then $\binom{n-x}{k-j}=0$. Hence we can rewrite the formula \eqref{eq04} as follows:
    \begin{equation*}
        \kappa^{(n)}_k(x)=\frac 1{\binom nk}\sum_{j=\max\{0,x+k-n\}}^{\min\{k,x\}}(-1)^j\binom xj\binom{n-x}{k-j}.
    \end{equation*}
    Using the formula above, for $x\in\{0,...,n\}$ we calculate
    \begin{equation*}
        \begin{split}
            \varphi&_t(x)=\sum_{k=0}^nu_t^k(1-u_t)^{n-k} \binom nk \kappa^{(n)}_k(x) =\sum_{k=0}^nu_t^k(1-u_t)^{n-k} \sum_{j=\max\{0,x+k-n\}}^{\min\{k,x\}}(-1)^j\binom xj\binom{n-x}{k-j}\\
            &=\sum_{j=0}^x(-1)^j\binom xj\sum_{k=j}^{j+n-x} u_t^k(1-u_t)^{n-k} \binom{n-x}{k-j} =\sum_{j=0}^x(-1)^j\binom xj\sum_{l=0}^{n-x} u_t^{l+j}(1-u_t)^{n-l-j} \binom{n-x}l\\
            &=\sum_{j=0}^x(-u_t)^j\binom xj(1-u_t)^{x-j} \underbrace{\sum_{l=0}^{n-x}u_t^l(1-u_t)^{n-x-l} \binom{n-x}l}_{\qquad\qquad\qquad\;\;\; =(u_t+1-u_t)^{n-x}=1} =(1-2u_t)^x=(e^{-t})^x.
        \end{split}
    \end{equation*}
\end{proof}

Lemma \ref{thm:inny_wzor_na_Nt} will help us to prove the following:

\begin{lemma}\label{thm:Nt_sym_polgr_dyfuzji}
    The indexed family of operators $(N_t:t\in [0,\infty))$ is a symmetric diffusion semigroup (in the sense of \cite[Chapter III]{Stein}).
\end{lemma}

This lemma is well-known. We present the proof for the reader's convenience.

\begin{proof}[Proof of Lemma \ref{thm:Nt_sym_polgr_dyfuzji}]
    We need to check that for every $t\in [0,\infty)$ the operator $N_t$ has the properties of contraction, symmetry, positivity, and conservation.\\
    {\bf Contraction property.} Using Lemma \ref{thm:inny_wzor_na_Nt} and Lemma \ref{thm:Sk_kontrakcja}, for any $p\in[1,\infty]$ and $f:\mathbb I^n\to\mathbb C$ we obtain
    \begin{equation*}
        \begin{split}
            \|N_tf\|_p&\leqslant \sum_{k=0}^n\binom nku_t^k(1-u_t)^{n-k} \|S_kf\|_p \leqslant \sum_{k=0}^n\binom nku_t^k(1-u_t)^{n-k} \|f\|_p=\|f\|_p.
        \end{split}
    \end{equation*}
    {\bf Symmetry property.} For $y,z\in\mathbb I^n$ we have
    \begin{equation*}
        \begin{split}
            \langle N_t\widetilde{\chi_y}, \widetilde{\chi_z}\rangle &=\langle e^{-t|y|} \widetilde{\chi_y}, \widetilde{\chi_z}\rangle = e^{-t|y|}\delta_{y,z} = e^{-t|z|}\delta_{y,z}= \langle\widetilde{\chi_y}, e^{-t|z|}\widetilde{\chi_z}\rangle = \langle \widetilde{\chi_y}, N_t\widetilde{\chi_z}\rangle ,
        \end{split}
    \end{equation*}
    where $\delta_{y,z}$ is the Kronecker delta. By linearity of $N_t$ and sesquilinearity of $\langle\cdot,\cdot\rangle$ we get $\langle N_tf,g\rangle =\langle f,N_tg\rangle$ for all functions $f,g:\mathbb I^n\to\mathbb C$.\\
    {\bf Positivity property.} For any function $f:\mathbb I^n\to [0,\infty)$ we have $S_kf=f*\frac 1{|R_k|}\mathds 1_{R_k}\geqslant 0$ so by Lemma \ref{thm:inny_wzor_na_Nt} we obtain
    \begin{equation*}
        N_tf=\sum_{k=0}^n\binom nku_t^k(1-u_t)^{n-k} S_kf\geqslant 0.
    \end{equation*}
    {\bf Conservation property.} Let us denote by $\mathbf 0_n$ the element of $\mathbb I^n$ consisting of only zeros. Then we have $\mathds 1_{\mathbb I^n}=\chi_{\mathbf 0_n}$ and hence
    \begin{equation*}
        N_t\mathds 1_{\mathbb I^n}=N_t\chi_{\mathbf 0_n}=e^{-t|\mathbf 0_n|}\chi_{\mathbf 0_n}=\mathds 1_{\mathbb I^n}.
    \end{equation*}
\end{proof}

\subsection{\texorpdfstring{The $r$-variation seminorm}{The r-variation seminorm}}
Let $r\in [1,\infty)$ and $Z\subseteq [0,\infty)$. For a function $Z\ni t\mapsto a_t\in\mathbb C$ we define the $r$-variation seminorm:
\begin{equation*}
    V_r(a_t:t\in Z):=\sup_{\begin{smallmatrix} t_0<...<t_J\\ t_j\in Z \end{smallmatrix}} \left( \sum_{j=1}^J|a_{t_{j-1}}-a_{t_j}|^r \right)^{\frac 1r},
\end{equation*}
where the supremum is taken over all finite nonempty increasing sequences in $Z$. The supremum is taken in the partially ordered set $[0,\infty]$ (hence $\sup\emptyset =0$). Of course, we could equivalently consider nondecreasing sequences $(t_0,...,t_J)$ instead of increasing ones. Anyway, we can assume that for every sequence $(t_0,...,t_J)$ which we take the supremum over we have $J\geqslant 1$ (since if $J=0$ then $\sum_{j=1}^J|a_{t_{j-1}}-a_{t_j}|^r=0$).\\
We collect here several properties of the $r$-variation seminorm:
\begin{itemize}
    \item If $1\leqslant r_1\leqslant r_2<\infty$ then
    \begin{equation}\label{eq07}
        V_{r_2}(a_t:t\in Z)\leqslant V_{r_1}(a_t:t\in Z).
    \end{equation}
    \item If $Y\subseteq Z$ then
    \begin{equation*}
        V_r(a_t:t\in Y)\leqslant V_r(a_t:t\in Z).
    \end{equation*}
    \item If $W\subseteq [0,\infty)$ and $\varphi$ is a nondecreasing function from $W$ into $Z$ then
    \begin{equation}\label{eq23}
        V_r(a_{\varphi(s)}:s\in W)=V_r(a_t:t\in\varphi[W]).
    \end{equation}
    \item Triangle inequality: if $t\mapsto b_t$ is a function from $Z$ into $\mathbb C$ then
    \begin{equation}\label{eq05}
        V_r(a_t+b_t:t\in Z)\leqslant V_r(a_t:t\in Z)+V_r(b_t:t\in Z).
    \end{equation}
    \item If the set $Z$ is countable then
    \begin{equation}\label{eq06}
        V_r(a_t:t\in Z)\leqslant 2\left( \sum_{t\in Z}|a_t|^r \right)^{\frac 1r}.
    \end{equation}
    \item Let us denote $\mathbb D:=\{2^l:l\in\mathbb Z\}$ and
    \begin{equation*}
        \lfloor t\rfloor_D:=\max\{k\in\mathbb D:k\leqslant t\} \qquad\text{for }t\in (0,\infty).
    \end{equation*}
    If $Z\subseteq (0,\infty)$ and $(\forall t\in Z)\lfloor t\rfloor_D\in Z$ then
    \begin{equation}\label{eq08}
        \begin{split}
            V_r(a_t:t\in Z)\lesssim &\, V_r(a_t:t\in\mathbb D\cap Z) +\left( \sum_{k\in\mathbb D}V_r(a_t:t\in [k,2k)\cap Z)^r \right)^{\frac 1r}.
        \end{split}
    \end{equation}
\end{itemize}
All these properties are standard but for the reader's convenience we will prove the last one:

\begin{proof}[Proof of the inequality \eqref{eq08}]
    Let us assume that $Z\subseteq (0,\infty)$ and $(\forall t\in Z)\lfloor t\rfloor_D\in Z$. %Let us denote
    % \begin{equation*}
    %     \mathbb D_Z:=\{\lfloor t\rfloor_D:t\in Z\}.
    % \end{equation*}
    % It is easy to check that
    % \begin{equation}\label{eq10}
    %     \mathbb D\cap Z=\mathbb D_Z.
    % \end{equation}
    Let us take any finite sequence $t_0\!<\!...\!<\!t_J,\;\; t_j\!\in\! Z$. Let
    \begin{equation*}
        B_1:=\big\{j\in\{1,...,J\}:\lfloor t_{j-1}\rfloor_D=\lfloor t_j\rfloor_D\big\}
    \end{equation*}
    and
    \begin{equation*}
        \begin{split}
            B_2:=&\;\big\{j\in\{1,...,J\}:\lfloor t_{j-1}\rfloor_D<\lfloor t_j\rfloor_D\big\}\\
            =&\;\{1,...,J\}\setminus B_1.
        \end{split}
    \end{equation*}
    We have
    \begin{equation}\label{eq09}
        \begin{split}
            \sum_{j\in B_1}|a_{t_{j-1}}-a_{t_j}|^r &=\sum_{k\in\mathbb D} \sum_{\begin{smallmatrix}
                j\in B_1\\
                \lfloor t_j\rfloor_D=k
            \end{smallmatrix}}|a_{t_{j-1}}-a_{t_j}|^r \leqslant\sum_{k\in\mathbb D}V_r(a_t:t\in [k,2k)\cap Z)^r.
        \end{split}
    \end{equation}
    Next, since for $j\in\{0,...,J\}$ we have $\lfloor t_j\rfloor_D\in Z$, we can write for $j\in B_2$ the following:
    \begin{equation*}
        a_{t_{j-1}}-a_{t_j}= (a_{t_{j-1}}-a_{\lfloor t_{j-1}\rfloor_D})+(a_{\lfloor t_{j-1}\rfloor_D}-a_{\lfloor t_j\rfloor_D})+(a_{\lfloor t_j\rfloor_D}-a_{t_j}),
    \end{equation*}
    and by Minkowski's inequality we get
    \begin{equation*}
        \begin{split}
            \Bigg( &\sum_{j\in B_2}|a_{t_{j-1}}-a_{t_j}|^r \Bigg)^{\frac 1r}\\
            &\leqslant \Bigg( \sum_{j\in B_2}|a_{t_{j-1}}-a_{\lfloor t_{j-1}\rfloor_D}|^r \Bigg)^{\frac 1r}+ \Bigg( \sum_{j\in B_2}|a_{\lfloor t_{j-1}\rfloor_D}-a_{\lfloor t_j\rfloor_D}|^r \Bigg)^{\frac 1r}+ \Bigg( \sum_{j\in B_2}|a_{\lfloor t_j\rfloor_D}-a_{t_j}|^r \Bigg)^{\frac 1r}\\
            &=:\mathcal I_1^{\frac 1r}+ \mathcal I_2^{\frac 1r}+ \mathcal I_3^{\frac 1r}.
        \end{split}
    \end{equation*}
    %Since $\lfloor t_j\rfloor_D\in Z$ for $j\in\{0,...,J\}$,
    We have
    \begin{equation*}
        \begin{split}
            \mathcal I_3 &\leqslant\sum_{j\in B_2}V_r\big(a_t:t\in \big[\lfloor t_j\rfloor_D,2\lfloor t_j\rfloor_D\big)\cap Z\big)^r \leqslant\sum_{k\in\mathbb D}V_r(a_t:t\in [k,2k)\cap Z)^r,
        \end{split}
    \end{equation*}
    where the last inequality follows from the observation that for $j_1,j_2\!\in\! B_2,\;\: j_1\!<\!j_2$, we have $t_{j_1}\!\leqslant\! t_{j_2-1}$ and $\lfloor t_{j_1}\rfloor_D\!\leqslant\!\lfloor t_{j_2-1}\rfloor_D\!<\! \lfloor t_{j_2}\rfloor_D$.\\
    Similarly we get
    \begin{equation*}
        \begin{split}
            \mathcal I_1 &\leqslant\sum_{j\in B_2}V_r\big(a_t:t\in \big[\lfloor t_{j-1}\rfloor_D,2\lfloor t_{j-1}\rfloor_D\big)\cap Z\big)^r \leqslant\sum_{k\in\mathbb D}V_r(a_t:t\in [k,2k)\cap Z)^r,
        \end{split}
    \end{equation*}
    where the last inequality follows from the observation that for $j_1,j_2\!\in\! B_2,\;\: j_1\!<\!j_2$, we have $t_{j_1}\!\leqslant\! t_{j_2-1}$ and $\lfloor t_{j_1-1}\rfloor_D\!<\! \lfloor t_{j_1}\rfloor_D\!\leqslant\!\lfloor t_{j_2-1}\rfloor_D$.\\
    Next, (if $B_2\neq\emptyset$) we can write $B_2=\{i_1,...,i_M\}$, where $i_1\!<\!...\!<\!i_M$. Then we have $t_{i_1-1} \!\leqslant t_{i_1}\!\leqslant t_{i_2-1} \!\leqslant t_{i_2} \!\leqslant ... \leqslant t_{i_M-1} \!\leqslant t_{i_M}$ and ${\lfloor t_{i_1-1}\rfloor_D\!\leqslant\!\lfloor t_{i_1}\rfloor_D\!\leqslant\! \lfloor t_{i_2-1}\rfloor_D\!\leqslant\!\lfloor t_{i_2}\rfloor_D\!\leqslant ... \leqslant\! \lfloor t_{i_M-1}\rfloor_D\!\leqslant\!\lfloor t_{i_M}\rfloor_D}$. Hence
    \begin{equation*}
        \mathcal I_2^{\frac 1r}=\left( \sum_{m=1}^M|a_{\lfloor t_{i_m-1}\rfloor_D}- a_{\lfloor t_{i_m}\rfloor_D}|^r \right)^{\frac 1r}\leqslant V_r(a_t:t\in\mathbb D\cap Z).
    \end{equation*}
    We obtain
    \begin{equation}\label{eq11}
        \begin{split}
            \left( \sum_{j\in B_2} |a_{t_{j-1}}-a_{t_j}|^r \right)^{\frac 1r} &\leqslant\mathcal I_1^{\frac 1r}+ \mathcal I_2^{\frac 1r}+ \mathcal I_3^{\frac 1r}\\
            &\leqslant 2\left( \sum_{k\in\mathbb D}V_r(a_t:t\in [k,2k)\cap Z)^r \right)^{\frac 1r} + V_r(a_t:t\in\mathbb D\cap Z).
        \end{split}
    \end{equation}
    By \eqref{eq09} and \eqref{eq11} we get
    \begin{equation*}
        \begin{split}
            \Bigg( \sum_{j=1}^J |a_{t_{j-1}}-a_{t_j}|^r \Bigg)^{\frac 1r} &\leqslant \left( \sum_{j\in B_1} |a_{t_{j-1}}-a_{t_j}|^r \right)^{\frac 1r} + \left( \sum_{j\in B_2} |a_{t_{j-1}}-a_{t_j}|^r \right)^{\frac 1r}\\
            &\leqslant 3\left( \sum_{k\in\mathbb D}V_r(a_t:t\in [k,2k)\cap Z)^r \right)^{\frac 1r} + V_r(a_t:t\in\mathbb D\cap Z).
        \end{split}
    \end{equation*}
    Hence
    \begin{equation*}
        V_r(a_t:t\in Z) \lesssim \left( \sum_{k\in\mathbb D}V_r(a_t:t\in [k,2k)\cap Z)^r \right)^{\frac 1r} + V_r(a_t:t\in\mathbb D\cap Z),
    \end{equation*}
    which is the inequality \eqref{eq08}. %in view of \eqref{eq10}.
\end{proof}

We will need two technical lemmas stated below. Lemma \ref{thm:rozbicie_przedzialu} is stated and proved in \linebreak
\cite[proof of Lemma 1]{MT}.

\begin{lemma}\label{thm:rozbicie_przedzialu}
    Let $l\in\{0,1,2,...\}$ and for $g\in\{0,...,l\}$ let
    \begin{equation*}
        \mathcal I_g:=\big\{ [(h\!-\!1)2^g,h2^g):h\in\{1,...,2^{l-g}\}\big\}.
    \end{equation*}
    For every $a,b\in\{0,...,2^l\}$ such that $a<b$ there exists a pairwise disjoint family $\mathcal A(a,b)\subseteq\bigcup_{g=0}^l\mathcal I_g$ such that
    \begin{equation*}
        [a,b)=\bigcup_{P\in\mathcal A(a,b)}P\qquad\text{and}\qquad (\forall g)\;|\mathcal I_g\cap\mathcal A(a,b)|\leqslant 2.
    \end{equation*}
\end{lemma}

The following lemma is a modification of \cite[Lemma 1]{MT} and has a similar proof. We provide details for the reader's convenience.

\begin{lemma}\label{thm:usuniecie_poln_war}
    Let $s\in [1,\infty),\;\;l\in\{0,1,2,...\}$, and $M\in\mathbb Z$. Let $k\mapsto a_k$ be a~function from $\{0,...,2^l\}\cap (-\infty,M]$ into $\mathbb C$. Then
    \begin{equation*}
        V_s\big( a_k:k\in \{0,...,2^l\}\cap (-\infty,M]\big)\leqslant 2^{1-\frac 1s}\sum_{g=0}^l\left( \sum_{\begin{smallmatrix}
            h\in\{1,...,2^{l-g}\}\\
            h2^g\leqslant M
        \end{smallmatrix}} |a_{(h-1)2^g}-a_{h2^g}|^s \right)^{\frac 1s}.
    \end{equation*}
\end{lemma}

\begin{proof}
    For $g\in\{0,...,l\}$ let $\mathcal I_g$ be defined as in Lemma \ref{thm:rozbicie_przedzialu}.\\
    Let us take any finite sequence $k_0\!<\!...\!<\!k_J,\;\; k_j\!\in\! \{0,...,2^l\}\cap (-\infty,M]$. For $j\in\{1,...,J\}$ let us choose a family $\mathcal A(k_{j-1},k_j)$ as in Lemma \ref{thm:rozbicie_przedzialu} and let us denote $\mathcal B_j:= \mathcal A(k_{j-1},k_j)$ for short. Let us note that
    \begin{equation}\label{eq25}
        \begin{split}
            &(\forall j_1,j_2)(j_1\neq j_2\implies\mathcal B_{j_1}\cap\mathcal B_{j_2}=\emptyset)\quad\text{ and}\\
            &(\forall j)(\forall P\in\mathcal B_j)\sup P\leqslant M.
        \end{split}
    \end{equation}
    For every $j\in\{1,...,J\}$ we have
    \begin{equation*}
        |a_{k_{j-1}}-a_{k_j}|\leqslant\sum_{P\in\mathcal B_j}|a_{\inf P}-a_{\sup P}|=\sum_{g=0}^l \sum_{P\in\mathcal B_j\cap \mathcal I_g}|a_{\inf P}-a_{\sup P}|.
    \end{equation*}
    By the inequality above and by Minkowski's inequality we get
    \begin{equation}\label{eq26}
        \begin{split}
            \Bigg(\sum_{j=1}^J |a_{k_{j-1}}-a_{k_j}|^s\Bigg)^{\frac 1s}&\leqslant \Bigg(\sum_{j=1}^J \Bigg( \sum_{g=0}^l \sum_{P\in\mathcal B_j\cap \mathcal I_g}|a_{\inf P}-a_{\sup P}| \Bigg)^s\Bigg)^{\frac 1s}\\
            &\leqslant \sum_{g=0}^l\Bigg(\sum_{j=1}^J \Bigg( \sum_{P\in\mathcal B_j\cap \mathcal I_g}|a_{\inf P}-a_{\sup P}| \Bigg)^s\Bigg)^{\frac 1s}.
        \end{split}
    \end{equation}
    For every $g$ and $j$ we have (by H\"older's inequality)
    \begin{equation*}
        \begin{split}
            \sum_{P\in\mathcal B_j\cap \mathcal I_g}|a_{\inf P}-a_{\sup P}|&\leqslant |\mathcal B_j\cap \mathcal I_g|^{1-\frac 1s}\Bigg( \sum_{P\in\mathcal B_j\cap \mathcal I_g}|a_{\inf P}-a_{\sup P}|^s\Bigg)^{\frac 1s}\\
            &\leqslant 2^{1-\frac 1s}\Bigg( \sum_{P\in\mathcal B_j\cap \mathcal I_g}|a_{\inf P}-a_{\sup P}|^s\Bigg)^{\frac 1s}.
        \end{split}
    \end{equation*}
    Hence by \eqref{eq26} we get
    \begin{equation}\label{eq27}
        \Bigg(\sum_{j=1}^J |a_{k_{j-1}}-a_{k_j}|^s\Bigg)^{\frac 1s}\leqslant 2^{1-\frac 1s} \sum_{g=0}^l\Bigg(\sum_{j=1}^J \sum_{P\in\mathcal B_j\cap \mathcal I_g}|a_{\inf P}-a_{\sup P}|^s\Bigg)^{\frac 1s}.
    \end{equation}
    For every $g$, by \eqref{eq25} we have
    \begin{equation*}
        \sum_{j=1}^J \sum_{P\in\mathcal B_j\cap \mathcal I_g}|a_{\inf P}-a_{\sup P}|^s\leqslant\sum_{\begin{smallmatrix}
            P\in\mathcal I_g\\
            \sup P\leqslant M
        \end{smallmatrix}} |a_{\inf P}-a_{\sup P}|^s= \sum_{\begin{smallmatrix}
            h\in\{1,...,2^{l-g}\}\\
            h2^g\leqslant M
        \end{smallmatrix}} |a_{(h-1)2^g}-a_{h2^g}|^s.
    \end{equation*}
    Hence by \eqref{eq27} we get
    \begin{equation*}
        \Bigg(\sum_{j=1}^J |a_{k_{j-1}}-a_{k_j}|^s\Bigg)^{\frac 1s}\leqslant 2^{1-\frac 1s} \sum_{g=0}^l\left( \sum_{\begin{smallmatrix}
            h\in\{1,...,2^{l-g}\}\\
            h2^g\leqslant M
        \end{smallmatrix}} |a_{(h-1)2^g}-a_{h2^g}|^s \right)^{\frac 1s}.
    \end{equation*}
    Since the sequence $(k_0,...,k_J)$ was arbitrary, we obtain the desired inequality.
\end{proof}

\begin{remark}
    If $f$ is a function from $\mathbb I^n$ into $\mathbb C$ then writing
    \begin{equation*}
        V_r(S_kf:k\in\{0,...,n\})
    \end{equation*}
    we mean the function
    \begin{equation*}
        \mathbb I^n\ni x\mapsto V_r(S_kf(x):k\in\{0,...,n\}).
    \end{equation*}
\end{remark}

\section{Estimates for the Krawtchouk polynomials and the counterexamples}

% \begin{remark}\label{thm:szac_dla_krawtch_z_dwojka}
%     In fact we have proved that if $n\geqslant 2$ then for $k,x\in\{0,...,n\}$ we have
%     \begin{equation*}
%         |\kappa^{(n)}_k(x)-1|\leqslant 2\frac{kx}n.
%     \end{equation*}
% \end{remark}

\begin{lemma}\label{thm:szac_dla_krawtch}
    Let $k,x\in\{0,...,n\}$.
    \begin{itemize}
        \item[(a)] We have
        \begin{equation*}
            |\kappa^{(n)}_k(x)-1|\leqslant 2\frac{kx}n.
        \end{equation*}
        \item[(b)] If $k,x\geqslant 1$ and $k,x\leqslant\frac n2$ then
        \begin{equation*}
            |\kappa^{(n)}_k(x)|\lesssim\frac n{kx}.
        \end{equation*}
        \item[(c)] If $k\geqslant 1$ and $k,x\leqslant\frac n2$ then
        \begin{equation*}
            |\kappa^{(n)}_k(x)-\kappa^{(n)}_{k-1}(x)| \lesssim\frac 1k.
        \end{equation*}
    \end{itemize}
\end{lemma}

\begin{proof}
    (a) It is easy to show the desired inequality in the case $n=1$. Therefore, we will consider only the case $n\geqslant 2$.\\
    Take any $z\in\{1,...,x\}$. If $k\geqslant 1$ then by Lemma \ref{thm:roznica_krawtch} we get
    \begin{equation*}
        |\kappa^{(n)}_k(z)-\kappa^{(n)}_k(z-1)|=2\frac kn|\kappa^{(n-1)}_{k-1}(z-1)|\leqslant 2\frac kn,
    \end{equation*}
    while if $k=0$ then
    \begin{equation*}
        |\kappa^{(n)}_k(z)-\kappa^{(n)}_k(z-1)|=|1-1|= 0\leqslant 2\frac kn.
    \end{equation*}
    Hence
    \begin{equation*}
        \begin{split}
            |1-\kappa^{(n)}_k(x)|&=|\kappa^{(n)}_k(0)- \kappa^{(n)}_k(x)|\leqslant\sum_{z=1}^x |\kappa^{(n)}_k(z-1)-\kappa^{(n)}_k(z)| \leqslant x\cdot 2\frac kn=2\frac{kx}n.
        \end{split}
    \end{equation*}

    \vspace{2ex}
    
    \noindent (b) The desired inequality follows easily from Lemma~\ref{thm:kappa_jak_e} and the inequality $e^{-t}\leqslant\frac 1t$ for $t>0$.

    \vspace{2ex}
    
    \noindent (c) It suffices to consider the case $n\geqslant 2$.\\
    If $x=0$ then
    \begin{equation*}
        |\kappa^{(n)}_k(x)-\kappa^{(n)}_{k-1}(x)|=|1-1|=0.
    \end{equation*}
    If $x\geqslant 1$ then by Lemma \ref{thm:roznica_krawtch} we can write
    \begin{equation*}
        |\kappa^{(n)}_k(x)-\kappa^{(n)}_{k-1}(x)|= |\kappa^{(n)}_x(k)-\kappa^{(n)}_x(k-1)|=2\frac xn|\kappa^{(n-1)}_{x-1}(k-1)|
    \end{equation*}
    and we consider three cases.\\
    $\underline{x=1}$. Then
    \begin{equation*}
        2\frac xn|\kappa^{(n-1)}_{x-1}(k-1)|=2\cdot\frac 1n\leqslant 2\cdot\frac 1k.
    \end{equation*}
    $\underline{x\geqslant 2\;\wedge\; k=1}$. Then
    \begin{equation*}
        2\frac xn|\kappa^{(n-1)}_{x-1}(k-1)|=2\frac xn\leqslant 2=2\cdot\frac 1k.
    \end{equation*}
    $\underline{x\geqslant 2\;\wedge\; k\geqslant 2}$. Then using Lemma \ref{thm:szac_dla_krawtch}(b) we get
    \begin{equation*}
        2\frac xn|\kappa^{(n-1)}_{x-1}(k-1)|\lesssim \frac xn\frac{n-1}{(x-1)(k-1)}=\frac{n-1}n \frac x{x-1}\frac k{k-1}\frac 1k\lesssim\frac 1k.
    \end{equation*}
    So for any $x\in\{0,...,\lfloor\frac n2\rfloor\}$ and $k\in\{1,...,\lfloor\frac n2\rfloor\}$ we obtain
    \begin{equation*}
        |\kappa^{(n)}_k(x)-\kappa^{(n)}_{k-1}(x)| \lesssim\frac 1k.
    \end{equation*}
\end{proof}

Let us denote by $\mathbf 1_n$ the element of $\mathbb I^n$ consisting of only ones.

\begin{counterexample}\label{thm:kontrprz_1}
    Let $r\in [1,\infty)$. Then
    \begin{equation*}
        \|V_r(S_k\chi_{\mathbf 1_n}:k\in\{0,...,n\})\|_2\geqslant 2n^{\frac 1r} \|\chi_{\mathbf 1_n}\|_2.
    \end{equation*}
\end{counterexample}

\begin{proof}
    Let us take any $x\in\mathbb I^n$. By the formula \eqref{eq47} and Fact \ref{thm:podst_wlasn_krawtch}, for any $k\in\{0,...,n\}$ we have
    \begin{equation*}
        S_k\chi_{\mathbf 1_n}(x)=\kappa^{(n)}_k(n)\chi_{\mathbf 1_n}(x)=(-1)^k\kappa^{(n)}_k(0)(-1)\!\!\!\!\overbrace{^{x\cdot \mathbf 1_n}}^{\quad\;\; =|x|}=(-1)^{k+|x|}.
    \end{equation*}
    Next,
    \begin{equation*}
        \begin{split}
            V_r(S_k\chi_{\mathbf 1_n}(x)&:k\in\{0,...,n\})\geqslant \left( \sum_{k=1}^n|S_{k-1}\chi_{\mathbf 1_n}(x)-S_k\chi_{\mathbf 1_n}(x)|^r \right)^{\frac 1r}\\
            &= \left( \sum_{k=1}^n| (-1)^{k-1+|x|}-(-1)^{k+|x|} |^r \right)^{\frac 1r}= \left( \sum_{k=1}^n 2^r \right)^{\frac 1r}= 2n^{\frac 1r}\mathds 1_{\mathbb I^n}(x).
        \end{split}
    \end{equation*}
    Hence
    \begin{equation*}
        \| V_r(S_k\chi_{\mathbf 1_n}:k\in\{0,...,n\})\|_2\geqslant \|2n^{\frac 1r}\mathds 1_{\mathbb I^n}\|_2=2n^{\frac 1r}\cdot\sqrt{2^n}= 2n^{\frac 1r}\|\chi_{\mathbf 1_n}\|_2.
    \end{equation*}
\end{proof}

\begin{counterexample}\label{thm:kontrprz_2}
    Let $r\in [1,\infty)$ and let $(a_m)_{m=1}^\infty$ be a sequence of numbers from $\big[\frac 13,\infty\big)$. For any $y\in\mathbb I^n$ such that $|y|\geqslant n-a_n$ we have
    \begin{equation*}
        \|V_r(S_k\chi_y:k\in\{0,...,n\})\|_2\geqslant \tfrac 23\Big\lfloor\tfrac n{3a_n}\Big\rfloor^{\frac 1r}\|\chi_y\|_2.
    \end{equation*}
\end{counterexample}

\begin{proof}
    Let us fix any $y\in\mathbb I^n$ such that $|y|\geqslant n-a_n$.\\
    For $k\in\{0,...,n\}$, by Lemma \ref{thm:szac_dla_krawtch}(a) we have
    \begin{equation*}
        \begin{split}
            |\kappa^{(n)}_k(|y|)-(-1)^k|=|(-1)^k \kappa^{(n)}_k(n-|y|)-(-1)^k|=| \kappa^{(n)}_k(n-|y|)-1| \leqslant 2\frac{k(n-|y|)}n\leqslant 2\frac{ka_n}n.
        \end{split}
    \end{equation*}
    Hence, if $k\leqslant \frac n{3a_n}$ then
    \begin{equation}\label{eq12}
        |\kappa^{(n)}_k(|y|)-(-1)^k|\leqslant\tfrac 23.
    \end{equation}
    For $k\in\big\{0,...,\big\lfloor \frac n{3a_n} \big\rfloor\big\}\setminus\big\{0\big\}$ we have $k\leqslant n$ (since $a_n\geqslant\frac 13$), and by triangle inequality and \eqref{eq12} we get
    \begin{equation*}
        \begin{split}
            2&=\big|(-1)^{k-1}-(-1)^k\big| \leqslant \big|(-1)^{k-1}\!-\!\kappa^{(n)}_{k-1}(|y|)\big|+\big| \kappa^{(n)}_{k-1}(|y|)\!-\! \kappa^{(n)}_k(|y|)\big|+\big| \kappa^{(n)}_k(|y|)\!-\!(-1)^k\big|\\
            &\leqslant \tfrac 23+ \big| \kappa^{(n)}_{k-1}(|y|)- \kappa^{(n)}_k(|y|)\big|+\tfrac 23
        \end{split}
    \end{equation*}
    and thus $\tfrac 23\leqslant \big| \kappa^{(n)}_{k-1}(|y|)- \kappa^{(n)}_k(|y|)\big|$.\\
    Let us take any $x\in\mathbb I^n$. For $k\in\big\{0,...,\big\lfloor \frac n{3a_n} \big\rfloor\big\}\setminus\big\{0\big\}$, by the last inequality we get
    \begin{equation*}
        \begin{split}
            \big| S_{k-1}\chi_y(x)-S_k\chi_y(x)\big|&=\big|\kappa^{(n)}_{k-1}(|y|)\chi_y(x)- \kappa^{(n)}_k(|y|)\chi_y(x)\big| =\big| \kappa^{(n)}_{k-1}(|y|)- \kappa^{(n)}_k(|y|)\big|\geqslant\tfrac 23.
        \end{split}
    \end{equation*}
    Therefore we obtain
    \begin{equation*}
        \begin{split}
            V_r(S_k\chi_y(x):k\in\{0,...,n\}) &\geqslant\left(\sum_{k=1}^{\left\lfloor\frac n{3a_n}\right\rfloor}\big|S_{k-1}\chi_y(x)-S_k\chi_y(x)\big|^r \right)^{\frac 1r}\\
            &\geqslant \left( \left\lfloor \tfrac n{3a_n} \right\rfloor\left( \tfrac 23 \right)^r \right)^{\frac 1r}=\tfrac 23 \left\lfloor \tfrac n{3a_n} \right\rfloor^{\frac 1r}\mathds 1_{\mathbb I^n}(x),
        \end{split}
    \end{equation*}
    and hence
    \begin{equation*}
        \begin{split}
            \| V_r(S_k\chi_y:k\in\{0,...,n\})\|_2 &\geqslant \left\| \tfrac 23 \left\lfloor \tfrac n{3a_n} \right\rfloor^{\frac 1r}\mathds 1_{\mathbb I^n} \right\|_2 = \tfrac 23 \left\lfloor \tfrac n{3a_n} \right\rfloor^{\frac 1r}\cdot\sqrt{2^n}= \tfrac 23 \left\lfloor \tfrac n{3a_n} \right\rfloor^{\frac 1r}\|\chi_y\|_2.
        \end{split}
    \end{equation*}
\end{proof}

\begin{remark}
    Counterexample \ref{thm:kontrprz_2} is of particular interest when the sequence $\big(\frac m{a_m}\big)_{m=1}^\infty$ is unbounded. In this case, it provides an example family of elements $y\in\mathbb I^n$ for which there is no dimension-free estimate of $\|V_r(S_k\chi_y:k\in\{0,...,n\})\|_2/\|\chi_y\|_2$. The condition $a_m\geqslant\frac 13$ is imposed purely for technical convenience. Indeed, if $(a_m)_{m=1}^\infty$ is any sequence of positive real numbers such that $\big(\frac m{a_m}\big)_{m=1}^\infty$ is unbounded, then the sequence $\Big(\frac m{\max\{\frac 13,a_m\}}\Big)_{m=1}^\infty$ remains unbounded. Furthermore, for $y\in\mathbb I^n$, we have $|y|\geqslant n-a_n$ if and only if $|y|\geqslant n-\max\{\frac 13,a_n\}$.
\end{remark}

\begin{corollary}\label{thm:wn_z_kontrprz_2}
    Let $r\in [1,\infty)$ and let $(b_n)_{n=1}^\infty$ be a sequence of positive real numbers such that $\big(\frac n{b_n}\big)_{n=1}^\infty$ is unbounded. Let us denote $E_n:=\{ y\in\mathbb I^n:|y|\geqslant n-b_n\}$ for $n\in\{1,2,3,...\}$. Then we have
    \begin{equation*}
        \sup_n \sup_{\begin{smallmatrix}
            f\in\mathbb C^{\mathbb I^n}\setminus\{ 0\}\\
            \widehat{\,f\,}\upharpoonright_{E_n}\equiv\, 0
        \end{smallmatrix}}\frac {\| V_r(S_kf:k\in\{0,...,n\})\|_2}{\|f\|_2} =\infty.
    \end{equation*}
\end{corollary}

\begin{proof}
    Let us denote $d_n:=\max\{\frac 19,b_n\}$. Then we have
    \begin{equation}\label{eq42}
        E_n=\{ y\in\mathbb I^n:|y|\geqslant n-d_n\}.
    \end{equation}
    Let us denote $a_n:=\sqrt{nd_n}$. Then $a_n\geqslant\sqrt{d_n}\geqslant\frac 13$ so by Counterexample \ref{thm:kontrprz_2} we get
    \begin{equation}\label{eq43}
        \frac{\|V_r(S_k\chi_y:k\in\{0,...,n\})\|_2}{\|\chi_y\|_2}\geqslant \tfrac 23\Big\lfloor\tfrac n{3a_n}\Big\rfloor^{\frac 1r}\qquad\text{for }|y|\geqslant n-a_n.
    \end{equation}
    Since the sequence $\big(\frac n{b_n}\big)_{n=1}^\infty$ is unbounded, it has a subsequence $\Big(\frac{n_l}{b_{n_l}}\Big)_{l=1}^\infty$ diverging to $\infty$. Then
    \begin{equation}\label{eq46}
        \frac{n_l}{d_{n_l}}=\frac{n_l}{\max\{\frac 19,b_{n_l}\}}=\min\left\{ 9n_l,\frac{n_l}{b_{n_l}} \right\}\xrightarrow[l\to\infty]{}\infty.
    \end{equation}
    Next, since $\frac{a_{n_l}}{d_{n_l}}=\sqrt{\frac{n_l}{d_{n_l}}} \xrightarrow[l\to\infty]{}\infty$, there exists $l_0$ such that for $l\geqslant l_0$ we have $\frac{a_{n_l}}{d_{n_l}}\geqslant 10$, and thus
    \begin{equation}\label{eq44}
        a_{n_l}-d_{n_l}=\left( \frac{a_{n_l}}{d_{n_l}}-1 \right) d_{n_l}\geqslant (10-1)\cdot\frac 19=1.
    \end{equation}
    Since $n_l-d_{n_l}=n_l\Big( 1-\frac{d_{n_l}}{n_l}\Big) \xrightarrow[l\to\infty]{}\infty$, there is $l_1$ such that for $l\geqslant l_1$ we have $\big\lceil n_l-d_{n_l}\big\rceil -1\geqslant 0$.\\
    Now for $l\geqslant\max\{l_0,l_1\}$ let us choose $y_l\in\mathbb I^{n_l}$ satisfying $|y_l|= \big\lceil n_l-d_{n_l}\big\rceil -1$. By \eqref{eq44} we get $|y_l|\geqslant n_l-d_{n_l}-1\geqslant n_l-a_{n_l}$, so by \eqref{eq43} we obtain
    \begin{equation}\label{eq45}
        \frac{\|V_r(S_k\chi_{y_l}:k\in\{0,...,n_l\})\|_2}{\|\chi_{y_l}\|_2}\geqslant \frac 23\left\lfloor\frac{n_l}{3a_{n_l}}\right\rfloor^{\frac 1r}.
    \end{equation}
    On the other hand we have $|y_l|<n_l-d_{n_l}$ so $y_l\not\in E_{n_l}$ by \eqref{eq42}. Hence by \eqref{eq01} we get $\widehat{\chi_{y_l}}\!\!\upharpoonright_{E_{n_l}}\equiv 0$ and therefore by \eqref{eq45} we obtain
    \begin{equation*}
        \begin{split}
            \sup_{\begin{smallmatrix}
                f\in\mathbb C^{\mathbb I^{n_l}}\setminus\{ 0\}\\
                \widehat{\,f\,}\upharpoonright_{E_{n_l}}\equiv\, 0
            \end{smallmatrix}}\frac {\| V_r(S_kf:k\in\{0,...,n_l\})\|_2}{\|f\|_2}&\geqslant \frac{\|V_r(S_k\chi_{y_l}:k\in\{0,...,n_l\})\|_2}{\|\chi_{y_l}\|_2}\geqslant \frac 23\left\lfloor\frac{n_l}{3a_{n_l}}\right\rfloor^{\frac 1r}\\
            &= \frac 23\left\lfloor\frac 13\sqrt{\frac{n_l}{d_{n_l}}}\right\rfloor^{\frac 1r} \xrightarrow[l\to\infty]{}\infty,
        \end{split}
    \end{equation*}
    where the divergence to $\infty$ follows from \eqref{eq46}. This implies the desired equality.
\end{proof}

\section{The main result}

The main result of this paper is the following theorem:

\begin{theorem}\label{thm:glowne_tw}
    Let $r\in (2,\infty)$ and $q\in\{0,1\}$. For any function $f:\mathbb I^n\to\mathbb C$ there holds the estimate
    \begin{equation*}
        \|V_r(S_kf:k\in (2\mathbb Z+q)\cap\{0,...,n\})\|_2\leqslant C_r\|f\|_2,
    \end{equation*}
    where $C_r>0$ is a constant dependent only on $r$ (in particular independent of the dimension $n$).
\end{theorem}

In order to prove the theorem above, it suffices to show the following

\begin{proposition}\label{thm:stw_1}
    Let $r\in (2,\infty)$. For any function $f:\mathbb I^n\to\mathbb C$ satisfying $\widehat f(y)=0$ for $|y|>\frac n2$ there holds
    \begin{equation*}
        \|V_r(S_kf:k\in\{0,...,n\})\|_2\lesssim_r\|f\|_2.
    \end{equation*}
\end{proposition}

\begin{proof}[Proof of the implication "Proposition \ref{thm:stw_1} $\Rightarrow$ Theorem \ref{thm:glowne_tw}"]\phantom{a}\\
    Let us take any $r\!\in\! (2,\infty),\;\; q\!\in\!\{0,1\}$, and any function $f:\mathbb I^n\to\mathbb C$. Let us denote
    \begin{equation*}
        f_1:=\sum_{\begin{smallmatrix}
            y\in\mathbb I^n\\
            |y|\leqslant\frac n2
        \end{smallmatrix}}\widehat f(y)\widetilde{\chi_y}\quad\text{and}\quad f_2:=\sum_{\begin{smallmatrix}
            y\in\mathbb I^n\\
            |y|>\frac n2
        \end{smallmatrix}}\widehat f(y)\widetilde{\chi_y}.
    \end{equation*}
    Obviously by \eqref{eq01} we have $f=f_1+f_2$. Let us take any $x\in\mathbb I^n$. By the triangle inequality \eqref{eq05} we get
    \begin{equation}\label{eq13}
        \begin{split}
            V_r&(S_kf(x):k\in (2\mathbb Z+q)\cap\{0,...,n\})^2\\
            &\leqslant\Big( V_r(S_kf_1(x):k\in (2\mathbb Z+q)\cap\{0,...,n\})+ V_r(S_kf_2(x):k\in (2\mathbb Z+q)\cap\{0,...,n\})\Big)^2\\
            &\lesssim V_r(S_kf_1(x):k\in (2\mathbb Z+q)\cap\{0,...,n\})^2+ V_r(S_kf_2(x):k\in (2\mathbb Z+q)\cap\{0,...,n\})^2.
        \end{split}
    \end{equation}
    Let us recall that $\mathbf 1_n$ denotes the element of $\mathbb I^n$ consisting of only ones. Let us define the linear operator $\theta$ on the space $\mathbb C^{\mathbb I^n}$ by
    \begin{equation*}
        \theta\chi_y:=\chi_{y\oplus \mathbf 1_n}\qquad\text{for }y\in\mathbb I^n.
    \end{equation*}
    For every $k\in\{0,...,n\}$ and $y,z\in\mathbb I^n$ we have
    \begin{equation*}
        \begin{split}
            &\chi_{y\oplus \mathbf 1_n}(z)=(-1)^{z\cdot (y\oplus \mathbf 1_n)}=(-1)^{z\cdot y+z\cdot \mathbf 1_n}=\chi_y(z)(-1)^{|z|},\\
            &S_k\chi_y(z)=\kappa^{(n)}_k(|y|)\chi_y(z)=(-1)^k\kappa^{(n)}_k(n-|y|)\chi_{y\oplus \mathbf 1_n}(z)(-1)^{|z|}\\
            &\qquad\quad\:\, =(-1)^{k+|z|}\kappa^{(n)}_k(|y\oplus \mathbf 1_n|)\chi_{y\oplus \mathbf 1_n}(z)=(-1)^{k+|z|}S_k\chi_{y\oplus \mathbf 1_n}(z) =(-1)^{k+|z|}S_k\theta\chi_y(z).
        \end{split}
    \end{equation*}
    Hence for every $k\in\{0,...,n\},\;\; g\in\mathbb C^{\mathbb I^n}$, and $z\in\mathbb I^n$ we get
    \begin{equation}\label{eq14}
        S_kg(z)=(-1)^{k+|z|}S_k\theta g(z).
    \end{equation}
    Let us take any finite sequence $k_0\!<\!...\!<\!k_J,\;\; k_j\in (2\mathbb Z+q)\cap\{0,...,n\}$. For every $j\in\{1,...,J\}$, by \eqref{eq14} we get
    \begin{equation*}
        \begin{split}
            |S_{k_{j-1}}&f_2(x)-S_{k_j}f_2(x)| =|(-1)^{k_{j-1}+|x|}S_{k_{j-1}}\theta f_2(x)-(-1)^{k_j+|x|}S_{k_j}\theta f_2(x)|\\
            &= |(-1)^{q+|x|}S_{k_{j-1}}\theta f_2(x)-(-1)^{q+|x|}S_{k_j}\theta f_2(x)| = |S_{k_{j-1}}\theta f_2(x)-S_{k_j}\theta f_2(x)|.
        \end{split}
    \end{equation*}
    Hence
    \begin{equation*}
        V_r(S_kf_2(x):k\in (2\mathbb Z+q)\cap\{0,...,n\})= V_r(S_k\theta f_2(x):k\in (2\mathbb Z+q)\cap\{0,...,n\}).
    \end{equation*}
    Combining the equality above with \eqref{eq13} we obtain
    \begin{equation*}
        \begin{split}
            V_r(S_k&f(x):k\in (2\mathbb Z+q)\cap\{0,...,n\})^2\\
            &\lesssim V_r(S_kf_1(x):k\in (2\mathbb Z+q)\cap\{0,...,n\})^2+ V_r(S_k\theta f_2(x):k\in (2\mathbb Z+q)\cap\{0,...,n\})^2\\
            &\leqslant V_r(S_kf_1(x):k\in \{0,...,n\})^2+ V_r(S_k\theta f_2(x):k\in \{0,...,n\})^2.
        \end{split}
    \end{equation*}
    Hence
    \begin{equation}\label{eq15}
        \begin{split}
            \|V_r(S_kf&:k\in (2\mathbb Z+q)\cap\{0,...,n\})\|_2^2=\sum_{x\in\mathbb I^n}V_r(S_kf(x):k\in (2\mathbb Z+q)\cap\{0,...,n\})^2\\
            &\lesssim\sum_{x\in\mathbb I^n}V_r(S_kf_1(x):k\in\{0,...,n\})^2+ \sum_{x\in\mathbb I^n}V_r(S_k\theta f_2(x):k\in\{0,...,n\})^2\\
            &=\| V_r(S_kf_1:k\in\{0,...,n\})\|_2^2+ \| V_r(S_k\theta f_2:k\in\{0,...,n\})\|_2^2.
        \end{split}
    \end{equation}
    Both functions $f_1$ and $\theta f_2$ have Fourier transform vanishing for arguments $y$ satisfying $|y|>\frac n2$ (this is because
    \begin{equation*}
        \theta f_2=\sum_{\begin{smallmatrix}
            y\in\mathbb I^n\\
            |y|>\frac n2
        \end{smallmatrix}} \widehat f(y)\theta\widetilde{\chi_y}=\sum_{\begin{smallmatrix}
            z\in\mathbb I^n\\
            |z|<\frac n2
        \end{smallmatrix}} \widehat f(z\oplus \mathbf 1_n)\widetilde{\chi_z}\;\;\; ).
    \end{equation*}
    Thus, assuming Proposition \ref{thm:stw_1}, by \eqref{eq15} we get
    \begin{equation*}
        \begin{split}
            \|V_r(S_kf:k\in (2\mathbb Z+q)\cap\{0,...,n\})\|_2^2 &\lesssim_r\|f_1\|_2^2+\|\theta f_2\|_2^2 =\sum_{\begin{smallmatrix}
                y\in\mathbb I^n\\
                |y|\leqslant\frac n2
            \end{smallmatrix}} |\widehat f(y)|^2+ \sum_{\begin{smallmatrix}
                y\in\mathbb I^n\\
                |y|>\frac n2
            \end{smallmatrix}} |\widehat f(y)|^2.
        \end{split}
    \end{equation*}
    Finally by Plancherel's equality \eqref{eq02} we obtain
    \begin{equation*}
        \sum_{\begin{smallmatrix}
                y\in\mathbb I^n\\
                |y|\leqslant\frac n2
            \end{smallmatrix}} |\widehat f(y)|^2+ \sum_{\begin{smallmatrix}
                y\in\mathbb I^n\\
                |y|>\frac n2
            \end{smallmatrix}} |\widehat f(y)|^2=\|f\|_2^2.
    \end{equation*}
\end{proof}

In order to prove Proposition \ref{thm:stw_1}, it suffices to show the following

\begin{proposition}\label{thm:stw_2}
    Let $r\in (2,\infty)$. For any function $f:\mathbb I^n\to\mathbb C$ satisfying $\widehat f(y)=0$ for $|y|>\frac n2$ there holds
    \begin{equation*}
        \|V_r(S_kf:k\in\{1,...,\big\lfloor\tfrac n2\big\rfloor\})\|_2\lesssim_r\|f\|_2.
    \end{equation*}
\end{proposition}

\begin{proof}[Proof of the implication "Proposition \ref{thm:stw_2} $\Rightarrow$ Proposition \ref{thm:stw_1}"]\phantom{a}\\
    Let us assume that Proposition \ref{thm:stw_2} holds. Let us fix any $r\in (2,\infty)$ and any function $f:\mathbb I^n\to\mathbb C$ satisfying $\widehat f(y)=0$ for $|y|>\frac n2$.\\
    Let us fix any $x\in\mathbb I^n$. Let us take any finite sequence $k_0<...<k_J$, where $k_j\in\{0,...,n\}$.\\
    If $k_0\leqslant\frac n2<k_J$ then there exists $J_0\in\{1,...,J\}$ such that $k_{J_0-1}\leqslant\frac n2$ and $k_{J_0}>\frac n2$. Then
    \begin{equation*}
        \begin{split}
            \Bigg(\sum_{j=1}^J |&S_{k_{j-1}}f(x)-S_{k_j}f(x)|^r\Bigg)^{\frac 1r} \leqslant \Bigg( \sum_{j=1}^{J_0-1} |S_{k_{j-1}}f(x)-S_{k_j}f(x)|^r\Bigg)^{\frac 1r}\\
            &+|S_{k_{J_0-1}}f(x)-S_{k_{J_0}}f(x)| +\Bigg( \sum_{j=J_0+1}^J |S_{k_{j-1}}f(x)-S_{k_j}f(x)|^r\Bigg)^{\frac 1r}\\
            \leqslant &\,\Bigg( \sum_{j=1}^{J_0-1} |S_{k_{j-1}}f(x)-S_{k_j}f(x)|^r\Bigg)^{\frac 1r} +|S_{k_{J_0-1}}f(x)-S_{\lfloor\frac n2\rfloor}f(x)| +|S_{\lfloor\frac n2\rfloor}f(x)|\\
            &+|\!-\!S_{\lfloor\frac n2\rfloor +1}f(x)|+|S_{\lfloor\frac n2\rfloor +1}f(x)-S_{k_{J_0}}f(x)| +\Bigg( \sum_{j=J_0+1}^J |S_{k_{j-1}}f(x)-S_{k_j}f(x)|^r\Bigg)^{\frac 1r}
        \end{split}
    \end{equation*}
    Using H\"older's inequality we continue to estimate
    \begin{equation*}
        \begin{split}
            \leqslant &\, 2^{1-\frac 1r} \Bigg( \sum_{j=1}^{J_0-1} |S_{k_{j-1}}f(x)-S_{k_j}f(x)|^r +|S_{k_{J_0-1}}f(x)-S_{\lfloor\frac n2\rfloor}f(x)|^r \Bigg)^{\frac 1r} + |S_{\lfloor\frac n2\rfloor}f(x)|\\
            &+|S_{\lfloor\frac n2\rfloor +1}f(x)| +2^{1-\frac 1r}\Bigg( |S_{\lfloor\frac n2\rfloor +1}f(x)-S_{k_{J_0}}f(x)|^r+\!\! \sum_{j=J_0+1}^J\!\! |S_{k_{j-1}}f(x)-S_{k_j}f(x)|^r\Bigg)^{\frac 1r}\\
            \lesssim_r\!\!&\;\;V_r(S_kf(x):k\in\{0,...,\big\lfloor\tfrac n2\big\rfloor\})+ |S_{\lfloor\frac n2\rfloor}f(x)|+|S_{\lfloor\frac n2\rfloor +1}f(x)| +V_r(S_kf(x):k\in\{\big\lfloor\tfrac n2\big\rfloor\! +\!1,...,n\}).
        \end{split}
    \end{equation*}
    Let us note that the inequality
    \begin{equation*}
        \begin{split}
            &\!\!\!\!\!\!\!\!\!\!\Bigg(\sum_{j=1}^J |S_{k_{j-1}}f(x)-S_{k_j}f(x)|^r\Bigg)^{\frac 1r}\\
            \lesssim_r&\,V_r(S_kf(x):k\in\{0,...,\big\lfloor\tfrac n2\big\rfloor\})+ |S_{\lfloor\frac n2\rfloor}f(x)|+|S_{\lfloor\frac n2\rfloor +1}f(x)| +V_r(S_kf(x):k\in\{\big\lfloor\tfrac n2\big\rfloor\! +\!1,...,n\})
        \end{split}
    \end{equation*}
    is true regardless of whether $k_0\leqslant\frac n2<k_J$ holds. Hence we have
    \begin{equation}\label{eq16}
        \begin{split}
            V_r(S_kf(x):k\in\{0,...,n\}) \lesssim_r&\,V_r(S_kf(x):k\in\{0,...,\big\lfloor\tfrac n2\big\rfloor\})+ |S_{\lfloor\frac n2\rfloor}f(x)|+|S_{\lfloor\frac n2\rfloor +1}f(x)|\\
            &+V_r(S_kf(x):k\in\{\big\lfloor\tfrac n2\big\rfloor\! +\!1,...,n\}).
        \end{split}
    \end{equation}
    Next, let us take any finite sequence $k_0\!<\!...\!<\!k_J,\;\; k_j\in\{0,...,\big\lfloor\frac n2\big\rfloor\}$.\\
    If $k_0=0$ then
    \begin{equation*}
        \begin{split}
            \Bigg(\sum_{j=1}^J |&S_{k_{j-1}}f(x)-S_{k_j}f(x)|^r\Bigg)^{\frac 1r} \leqslant |S_0f(x)-S_{k_1}f(x)|+\Bigg(\sum_{j=2}^J|S_{k_{j-1}}f(x)-S_{k_j}f(x)|^r\Bigg)^{\frac 1r}\\
            &\!\!\!\!\leqslant |S_0f(x)|+|\!-\!S_1f(x)|+|S_1f(x)-S_{k_1}f(x)| + \Bigg(\sum_{j=2}^J|S_{k_{j-1}}f(x)-S_{k_j}f(x)|^r\Bigg)^{\frac 1r}\qquad\quad
        \end{split}
    \end{equation*}
    Using H\"older's inequality we continue to estimate
    \begin{equation*}
        \begin{split}
            \leqslant &\, |S_0f(x)|+|S_1f(x)| +2^{1-\frac 1r}\Bigg(|S_1f(x)-S_{k_1}f(x)|^r+\sum_{j=2}^J|S_{k_{j-1}}f(x)-S_{k_j}f(x)|^r\Bigg)^{\frac 1r}\\
            \lesssim_r\!\!&\;\;\,|S_0f(x)|+|S_1f(x)|+V_r(S_kf(x):k\in\{1,...,\big\lfloor\tfrac n2\big\rfloor\}).
        \end{split}
    \end{equation*}
    Let us note that the inequality
    \begin{equation*}
        \begin{split}
            \Bigg(\sum_{j=1}^J |S_{k_{j-1}}f(x)-S_{k_j}f(x)|^r\Bigg)^{\frac 1r} \lesssim_r&\;|S_0f(x)|+|S_1f(x)|+V_r(S_kf(x):k\in\{1,...,\big\lfloor\tfrac n2\big\rfloor\})
        \end{split}
    \end{equation*}
    is true regardless of whether $k_0=0$. Hence we have
    \begin{equation}\label{eq17}
        \begin{split}
            V_r(S_kf(x):k\in\{0,...,\big\lfloor\tfrac n2\big\rfloor\}) \lesssim_r&\;|S_0f(x)|+|S_1f(x)|+V_r(S_kf(x):k\in\{1,...,\big\lfloor\tfrac n2\big\rfloor\}).
        \end{split}
    \end{equation}
    Similarly we can show that
    \begin{equation}\label{eq18}
        \begin{split}
            V_r(S_kf(x)&:k\in\{\big\lfloor\tfrac n2\big\rfloor\!+\!1,...,n\})\\
            &\lesssim_rV_r(S_kf(x):k\in\{ \big\lfloor\tfrac n2\big\rfloor\!+\!1,...,n\!-\!1\})+|S_{n-1}f(x)|+|S_nf(x)|.
        \end{split}
    \end{equation}
    Combining \eqref{eq16}, \eqref{eq17}, and \eqref{eq18} we obtain
    \begin{equation*}
        \begin{split}
            V_r(S_kf(x):k&\in\{0,...,n\})\\
            \lesssim_r\!&\sum_{k\in\{0,1,\lfloor\frac n2\rfloor, \lfloor\frac n2\rfloor+1,n-1,n\}} |S_kf(x)|\\
            &+ V_r(S_kf(x):k\!\in\!\{1,...,\big\lfloor\tfrac n2\big\rfloor\})+ V_r(S_kf(x):k\!\in\!\{ \big\lfloor\tfrac n2\big\rfloor\!+\!1,...,n\!-\!1\})
        \end{split}
    \end{equation*}
    so, since $x$ was arbitrary, we get
    \begin{equation*}
        \begin{split}
            \|V_r(S_kf:k&\in\{0,...,n\})\|_2\\
            \lesssim_r\!&\sum_{k\in\{0,1,\lfloor\frac n2\rfloor, \lfloor\frac n2\rfloor+1,n-1,n\}} \|S_kf\|_2\\
            &+ \|V_r(S_kf:k\!\in\!\{1,...,\big\lfloor\tfrac n2\big\rfloor\})\|_2+ \|V_r(S_kf:k\!\in\!\{ \big\lfloor\tfrac n2\big\rfloor\!+\!1,...,n\!-\!1\})\|_2\\
            =:\!&\;\,\mathcal I_1+\mathcal I_2+\mathcal I_3.
        \end{split}
    \end{equation*}
    By Lemma \ref{thm:Sk_kontrakcja} we get $\mathcal I_1\leqslant 6\|f\|_2$. Moreover we have $\mathcal I_2\lesssim_r\|f\|_2$ by the assumption that Proposition \ref{thm:stw_2} holds. Next, for every $k\in\{0,...,n\}$ and $x\in\mathbb I^n$ we have
    \begin{equation*}
        \begin{split}
            S_kf(x\oplus \mathbf 1_n)&=\frac 1{|R_k|}\sum_{y\in R_k} f(x\oplus \mathbf 1_n\oplus y) =\frac 1{|R_{n-k}|}\sum_{w\in R_{n-k}}f(x\oplus w)=S_{n-k}f(x)
        \end{split}
    \end{equation*}
    (see \cite[page 63]{HKS}). Now let us take any $x\in\mathbb I^n$ and any finite sequence $k_0\!<\!...\!<\!k_J,\;\; k_j\!\in\!\{\big\lfloor\tfrac n2\big\rfloor\!+\!1,...,n\!-\!1\}$. We have
    \begin{equation*}
        \begin{split}
            \Bigg(\sum_{j=1}^J & \,|S_{k_{j-1}}f(x\oplus \mathbf 1_n)-S_{k_j}f(x\oplus \mathbf 1_n)|^r\Bigg)^{\frac 1r} =\Bigg(\sum_{j=1}^J |S_{n-k_{j-1}}f(x)-S_{n-k_j}f(x)|^r\Bigg)^{\frac 1r}\\
            &\leqslant V_r(S_kf(x):k\!\in\!\{1,...,\big\lceil\tfrac n2\big\rceil\!-\!1\})\leqslant V_r(S_kf(x):k\!\in\!\{1,...,\big\lfloor\tfrac n2\big\rfloor\}).
        \end{split}
    \end{equation*}
    Thus
    \begin{equation*}
        V_r(S_kf(x\oplus \mathbf 1_n):k\!\in\!\{\big\lfloor\tfrac n2\big\rfloor\!+\!1,...,n\!-\!1\})\leqslant V_r(S_kf(x):k\!\in\!\{1,...,\big\lfloor\tfrac n2\big\rfloor\}).
    \end{equation*}
    Hence we obtain
    \begin{equation*}
        \begin{split}
            \mathcal I_3&=\Bigg(\sum_{x\in\mathbb I^n} V_r(S_kf(x\oplus \mathbf 1_n):k\!\in\!\{\big\lfloor\tfrac n2\big\rfloor\!+\!1,...,n\!-\!1\})^2\Bigg)^{\frac 12}\\
            &\leqslant \Bigg(\sum_{x\in\mathbb I^n} V_r(S_kf(x):k\!\in\!\{1,...,\big\lfloor\tfrac n2\big\rfloor\})^2\Bigg)^{\frac 12}= \mathcal I_2 \lesssim_r\|f\|_2.
        \end{split}
    \end{equation*}
    Finally we get
    \begin{equation*}
        \|V_r(S_kf:k\in\{0,...,n\})\|_2\lesssim_r\mathcal I_1+\mathcal I_2+\mathcal I_3\lesssim_r\|f\|_2.
    \end{equation*}
\end{proof}

\begin{proof}[Proof of Proposition \ref{thm:stw_2}]\phantom{a}\\
    Let us fix any $r\in (2,\infty)$ and any function $f:\mathbb I^n\to\mathbb C$ satisfying $\widehat f(y)=0$ for $|y|>\frac n2$.\\
    For every $x\in\mathbb I^n$, using the inequality \eqref{eq08} we get
    \begin{equation*}
        \begin{split}
            &\!\!\!\!\!\! V_r(S_kf(x):k\in \{ 1,...,\big\lfloor\tfrac n2\big\rfloor \})\\
            \lesssim &\, V_r(S_kf(x):k\in\mathbb D\cap \{ 1,...,\big\lfloor\tfrac n2\big\rfloor \} ) +\left( \sum_{m\in\mathbb D}V_r( S_kf(x) :k\in [m,2m)\cap \{ 1,...,\big\lfloor\tfrac n2\big\rfloor \} )^r \right)^{\frac 1r}.
        \end{split}
    \end{equation*}
    Hence
    \begin{equation}\label{eq31}
        \begin{split}
            &\!\!\!\!\!\!\|V_r(S_kf:k\in \{ 1,...,\big\lfloor\tfrac n2\big\rfloor \})\|_2\\
            \lesssim &\:\|V_r(S_kf:k\in\mathbb D\cap \{ 1,...,\big\lfloor\tfrac n2\big\rfloor \} )\|_2 +\Bigg\|\left( \sum_{m\in\mathbb D}V_r( S_kf :k\in [m,2m)\cap \{ 1,...,\big\lfloor\tfrac n2\big\rfloor \} )^r \right)^{\frac 1r}\Bigg\|_2\\
            &\!\!\!\!\!\! =:\mathcal I_1+\mathcal I_2.
        \end{split}
    \end{equation}
    In order to complete the proof, it suffices to estimate $\mathcal I_1$ and $\mathcal I_2$ by $\|f\|_2$ up to some multiplicative positive constants dependent only on $r$.\\
    Let us first observe that for $\alpha,\beta\in (0,\infty)$ we have
    \begin{equation*}
        \sum_{l\in\mathbb Z}\min\Big\{\alpha 2^l,\frac 1{\alpha 2^l}\Big\}^\beta =\sum_{\begin{smallmatrix}
            l\in\mathbb Z\\
            \alpha 2^l\leqslant 1
        \end{smallmatrix}} \big(\alpha 2^l\big)^\beta + \sum_{\begin{smallmatrix}
            l\in\mathbb Z\\
            \alpha 2^l>1
        \end{smallmatrix}} \Big(\frac 1{\alpha 2^l}\Big)^\beta
    \end{equation*}
    We denote $L:=\max\{ l\in\mathbb Z:\alpha 2^l\leqslant 1\}$ and estimate
    \begin{equation*}
        \begin{split}
            \sum_{l\in\mathbb Z}\min\Big\{\alpha 2^l,\frac 1{\alpha 2^l}\Big\}^\beta &=\sum_{m=0}^\infty \big(\alpha 2^{L-m}\big)^\beta +\sum_{m=0}^\infty \Big(\frac 1{\alpha 2^{L+1+m}}\Big)^\beta\\
            &=\underbrace{\big(\alpha 2^L\big)^\beta}_{\quad\leqslant 1} \sum_{m=0}^\infty 2^{-m\beta}+\underbrace{\Big(\frac 1{\alpha 2^{L+1}}\Big)^\beta}_{\quad\leqslant 1} \sum_{m=0}^\infty 2^{-m\beta}\leqslant\frac 2{1-2^{-\beta}}.
        \end{split}
    \end{equation*}
    Hence we get
    \begin{equation}\label{eq32}
        \sum_{l\in\mathbb Z}\min\Big\{\alpha 2^l,\frac 1{\alpha 2^l}\Big\}^\beta\lesssim_\beta 1\qquad\text{for }\alpha,\beta\in (0,\infty).
    \end{equation}
    Now we move on to obtaining the estimates $\mathcal I_1\lesssim_r\|f\|_2$ and $\mathcal I_2\lesssim_r\|f\|_2$.
    
    \phantom{a}\\
    {\bf Estimate of $\mathcal I_1$.}\\
    Let us recall that $\mathcal I_1$ was defined in \eqref{eq31}. Using the triangle inequality \eqref{eq05} and the inequality \eqref{eq06} we get
    \begin{equation}\label{eq19}
        \begin{split}
            \mathcal I_1\leqslant &\;\Big\|V_r\Big(N_{\frac kn}f:k\in\mathbb D\cap\big\{1,...,\big\lfloor\tfrac n2\big\rfloor\big\}\Big)\Big\|_2 +\Big\|V_r\Big(\Big(S_k-N_{\frac kn}\Big)f:k\in\mathbb D\cap\big\{1,...,\big\lfloor\tfrac n2\big\rfloor\big\}\Big)\Big\|_2\\
            \leqslant &\;\Big\|V_r\Big(N_{\frac kn}f:k\in\mathbb D\cap\big\{1,...,\big\lfloor\tfrac n2\big\rfloor\big\}\Big)\Big\|_2 +2\Bigg\|\Bigg(\sum_{k\in\mathbb D\cap\{1,...,\lfloor\frac n2\rfloor\}} \Big|\Big(S_k-N_{\frac kn}\Big)f\Big|^r\Bigg)^{\frac 1r}\Bigg\|_2\\
            \leqslant &\;\Big\|V_r\Big(N_{\frac kn}f:k\in\mathbb D\cap\big\{1,...,\big\lfloor\tfrac n2\big\rfloor\big\}\Big)\Big\|_2 +2\Bigg\|\Bigg(\sum_{k\in\mathbb D\cap\{1,...,\lfloor\frac n2\rfloor\}} \Big|\Big(S_k-N_{\frac kn}\Big)f\Big|^2\Bigg)^{\frac 12}\Bigg\|_2.
        \end{split}
    \end{equation}
    By Lemma \ref{thm:Nt_sym_polgr_dyfuzji} we can apply \cite[Theorem 3.3]{JR} (see also \cite[(5.3)]{BMSW}) to $(N_t:t\in [0,\infty))$ and we obtain
    \begin{equation*}
        \big\|V_r\big(N_tf:t\in (0,\infty)\big)\big\|_2\lesssim_r\|f\|_2.
    \end{equation*}
    Hence using \eqref{eq23} we get
    \begin{equation}\label{eq20}
        \begin{split}
            \Big\|V_r\Big(N_{\frac kn}f:k\in\mathbb D\cap\big\{1,...,\big\lfloor\tfrac n2\big\rfloor\big\}\Big)\Big\|_2&= \Big\|V_r\Big(N_tf:t\in\tfrac 1n\big(\mathbb D\cap\big\{1,...,\big\lfloor\tfrac n2\big\rfloor\big\}\big)\Big)\Big\|_2 \\
            &\leqslant\big\|V_r\big(N_tf:t\in (0,\infty)\big)\big\|_2\lesssim_r\|f\|_2.
        \end{split}
    \end{equation}
    Now we will show that
    \begin{equation}\label{eq21}
        \Bigg\|\Bigg(\sum_{k\in\mathbb D\cap\{1,...,\lfloor\frac n2\rfloor\}} \Big|\Big(S_k-N_{\frac kn}\Big)f\Big|^2\Bigg)^{\frac 12}\Bigg\|_2^2\lesssim\|f\|_2^2.
    \end{equation}
    We calculate
    \begin{equation*}
        \begin{split}
            \Bigg\|\Bigg(\sum_{k\in\mathbb D\cap\{1,...,\lfloor\frac n2\rfloor\}} \Big|\Big(S_k-N_{\frac kn}\Big)f\Big|^2\Bigg)^{\frac 12}\Bigg\|_2^2 &=\sum_{x\in\mathbb I^n}\; \sum_{k\in\mathbb D\cap\{1,...,\lfloor\frac n2\rfloor\}} \Big|\Big(S_k-N_{\frac kn}\Big)f(x)\Big|^2\\
            %&= \sum_{k\in\mathbb D\cap\{1,...,\lfloor\frac n2\rfloor\}}\; \sum_{x\in\mathbb I^n} \Big|\Big(S_k-N_{\frac kn}\Big)f(x)\Big|^2\\
            &= \sum_{k\in\mathbb D\cap\{1,...,\lfloor\frac n2\rfloor\}} \Big\|\Big(S_k-N_{\frac kn}\Big)f\Big\|_2^2.
        \end{split}
    \end{equation*}
    Using \eqref{eq01}, \eqref{eq47}, and \eqref{eq48}, and next the orthonormality of the system $(\widetilde{\chi_y}:y\in\mathbb I^n)$, we get
    \begin{equation}\label{eq22}
        \begin{split}
            &\!\Bigg\|\Bigg(\sum_{k\in\mathbb D\cap\{1,...,\lfloor\frac n2\rfloor\}} \Big|\Big(S_k-N_{\frac kn}\Big)f\Big|^2\Bigg)^{\frac 12}\Bigg\|_2^2 = \sum_{k\in\mathbb D\cap\{1,...,\lfloor\frac n2\rfloor\}} \left\| \sum_{\begin{smallmatrix}
                y\in\mathbb I^n\\
                |y|\leqslant\frac n2
            \end{smallmatrix}}\widehat f(y)\Big(\kappa^{(n)}_k(|y|)-e^{-\frac{k|y|}n}\Big)\widetilde{\chi_y} \right\|_2^2\\
            &= \sum_{k\in\mathbb D\cap\{1,...,\lfloor\frac n2\rfloor\}}\; \sum_{\begin{smallmatrix}
                y\in\mathbb I^n\\
                |y|\leqslant\frac n2
            \end{smallmatrix}} \Big| \widehat f(y)\Big(\kappa^{(n)}_k(|y|)-e^{-\frac{k|y|}n}\Big) \Big|^2 = \sum_{\begin{smallmatrix}
                y\in\mathbb I^n\\
                |y|\leqslant\frac n2
            \end{smallmatrix}} \Big|\widehat f(y)\Big|^2\Phi(|y|),
        \end{split}
    \end{equation}
    where
    \begin{equation*}
        \Phi(x):= \sum_{k\in\mathbb D\cap\{1,...,\lfloor\frac n2\rfloor\}} \Big| \kappa^{(n)}_k(x)-e^{-\frac{kx}n} \Big|^2\qquad\text{for }x\in\big\{0,...,\big\lfloor\tfrac n2\big\rfloor\big\}.
    \end{equation*}
    In view of the equality \eqref{eq22} and the equality
    \begin{equation*}
        \|f\|_2^2= \sum_{\begin{smallmatrix}
                y\in\mathbb I^n\\
                |y|\leqslant\frac n2
            \end{smallmatrix}} \Big|\widehat f(y)\Big|^2,
    \end{equation*}
    in order to prove \eqref{eq21} it suffices to show that uniformly
    \begin{equation}\label{eq30}
        \Phi(x)\lesssim 1\qquad\text{for } x\in\{0,...,\big\lfloor\tfrac n2\big\rfloor\}.
    \end{equation}
    Let us take any $x\in\{0,...,\big\lfloor\tfrac n2\big\rfloor\}$. If $x=0$ then we easily see that $\Phi(x)=0$. So now let us assume that $x>0$.\\
    For any $k\in\mathbb D\cap\{1,...,\big\lfloor\frac n2\big\rfloor\}$, by Lemma \ref{thm:szac_dla_krawtch}(a) we get
    \begin{equation}\label{eq33}
        \Big| \kappa^{(n)}_k(x)-e^{-\frac{kx}n} \Big|\leqslant \Big| \kappa^{(n)}_k(x)-1\Big|+\Big|1-e^{-\frac{kx}n} \Big|\lesssim\frac{kx}n
    \end{equation}
    while by Lemma \ref{thm:szac_dla_krawtch}(b) we get
    \begin{equation}\label{eq34}
        \Big| \kappa^{(n)}_k(x)-e^{-\frac{kx}n} \Big|\leqslant \Big| \kappa^{(n)}_k(x)\Big|+e^{-\frac{kx}n}\lesssim\frac n{kx}.
    \end{equation}
    Now by \eqref{eq33} and \eqref{eq34} we get
    \begin{equation*}
        \Phi(x)\lesssim\sum_{k\in\mathbb D\cap\{1,...,\lfloor\frac n2\rfloor\}} \min\Big\{\frac{kx}n,\frac n{kx}\Big\}^2\leqslant\sum_{l\in\mathbb Z} \min\Big\{\frac{2^lx}n,\frac n{2^lx}\Big\}^2\lesssim 1,
    \end{equation*}
    where the last inequality follows from \eqref{eq32}.\\
    Thus \eqref{eq30} and \eqref{eq21} are proved. Combining \eqref{eq19}, \eqref{eq20}, and \eqref{eq21} we obtain
    \begin{equation*}
        \mathcal I_1\lesssim_r\|f\|_2.
    \end{equation*}

    \phantom{a}\\
    {\bf Estimate of $\mathcal I_2$.}\\
    Let us recall that $\mathcal I_2$ was defined in \eqref{eq31}. Using the inequality \eqref{eq07} we get
    \begin{equation*}
        \begin{split}
            \mathcal I_2&\leqslant\Bigg\|\Bigg( \sum_{m\in\mathbb D}V_2\Big( S_kf :k\in \big[m,2m\big)\cap \big\{ 1,...,\big\lfloor\tfrac n2\big\rfloor\big\} \Big)^r \Bigg)^{\frac 1r}\Bigg\|_2\\
            &\leqslant\Bigg\|\Bigg( \sum_{m\in\mathbb D}V_2\Big( S_kf :k\in \big[m,2m\big)\cap \big\{ 1,...,\big\lfloor\tfrac n2\big\rfloor \big\} \Big)^2 \Bigg)^{\frac 12}\Bigg\|_2\\
            &= \Bigg\|\Bigg( \sum_{l=0}^\infty V_2\Big( S_kf :k\in \big[2^l,2^{l+1}\big)\cap \big\{ 1,...,\big\lfloor\tfrac n2\big\rfloor \big\} \Big)^2 \Bigg)^{\frac 12}\Bigg\|_2\\
            &= \Bigg\|\Bigg( \sum_{l=0}^\infty V_2\Big( S_kf :k\in \big\{2^l,...,2^{l+1}\!-\!1\big\}\cap \big(\!-\!\infty,\big\lfloor\tfrac n2\big\rfloor \big] \Big)^2 \Bigg)^{\frac 12}\Bigg\|_2.
        \end{split}
    \end{equation*}
    Hence, using \eqref{eq23} we get
    \begin{equation}\label{eq24}
        \begin{split}
            \mathcal I_2^2&\leqslant \Bigg\|\Bigg( \sum_{l=0}^\infty V_2\Big( S_{k'+2^l}f :k'\in \big\{0,...,2^l\!-\!1\big\}\cap \big(\!-\!\infty,\big\lfloor\tfrac n2\big\rfloor\!-\!2^l \big] \Big)^2 \Bigg)^{\frac 12}\Bigg\|_2^2\\
            &= \sum_{x\in\mathbb I^n} \sum_{l=0}^\infty V_2\Big( S_{k+2^l}f(x) :k\in \big\{0,...,2^l\!-\!1\big\}\cap \big(\!-\!\infty,\big\lfloor\tfrac n2\big\rfloor\!-\!2^l \big] \Big)^2.
        \end{split}
    \end{equation}
    For every $x\in\mathbb I^n$ and $l\in\{0,1,2,...\}$, by Lemma \ref{thm:usuniecie_poln_war} we get
    \begin{equation*}
        \begin{split}
            V_2\Big(S_{k+2^l}&f(x) :k\in \big\{0,...,2^l\big\}\cap \big(\!-\!\infty,\big\lfloor\tfrac n2\big\rfloor\!-\!2^l \big] \Big)\\
            &\leqslant\sqrt 2 \sum_{g=0}^l\left( \sum_{\begin{smallmatrix}
                h\in\{1,...,2^{l-g}\}\\
                h2^g\leqslant \lfloor\frac n2\rfloor -2^l
            \end{smallmatrix}} \big|S_{(h-1)2^g+2^l}f(x)-S_{h2^g+2^l}f(x)\big|^2 \right)^{\frac 12}\\
            &=\sqrt 2\sum_{g'=0}^l 2^{-\frac 14g'}\left( 2^{\frac 14g'}\left( \sum_{\begin{smallmatrix}
                h\in\{1,...,2^{g'}\}\\
                h2^{l-g'}+2^l\leqslant \lfloor\frac n2\rfloor
            \end{smallmatrix}} \big|\big(S_{(h-1)2^{l-g'}+2^l}-S_{h2^{l-g'}+2^l}\big)f(x)\big|^2 \right)^{\frac 12}\right)
        \end{split}
    \end{equation*}
    Using Schwarz inequality we continue to estimate
    \begin{equation*}
        \begin{split}
            &\leqslant\sqrt 2\Bigg(\sum_{g=0}^l 2^{-\frac 12g}\Bigg)^{\frac 12}\left( \sum_{g=0}^l 2^{\frac 12g} \sum_{\begin{smallmatrix}
                h\in\{1,...,2^g\}\\
                h2^{l-g}+2^l\leqslant \lfloor\frac n2\rfloor
            \end{smallmatrix}} \big|\big(S_{(h-1)2^{l-g}+2^l}-S_{h2^{l-g}+2^l}\big)f(x)\big|^2 \right)^{\frac 12}\\
            &\lesssim \left( \sum_{g=0}^l 2^{\frac 12g} \sum_{\begin{smallmatrix}
                h\in\{1,...,2^g\}\\
                h2^{l-g}+2^l\leqslant \lfloor\frac n2\rfloor
            \end{smallmatrix}} \big|\big(S_{(h-1)2^{l-g}+2^l}-S_{h2^{l-g}+2^l}\big)f(x)\big|^2 \right)^{\frac 12}.
        \end{split}
    \end{equation*}
    Hence by \eqref{eq24}, we obtain
    \begin{equation*}
        \begin{split}
            \mathcal I_2^2 &\lesssim\sum_{x\in\mathbb I^n}\sum_{l=0}^\infty \sum_{g=0}^l 2^{\frac 12g} \sum_{\begin{smallmatrix}
                h\in\{1,...,2^g\}\\
                h2^{l-g}+2^l\leqslant \lfloor\frac n2\rfloor
            \end{smallmatrix}} \big|\big(S_{(h-1)2^{l-g}+2^l}-S_{h2^{l-g}+2^l}\big)f(x)\big|^2\\
            &= \sum_{l=0}^\infty \sum_{g=0}^l 2^{\frac 12g} \sum_{\begin{smallmatrix}
                h\in\{1,...,2^g\}\\
                h2^{l-g}+2^l\leqslant \lfloor\frac n2\rfloor
            \end{smallmatrix}} \big\|\big(S_{(h-1)2^{l-g}+2^l}-S_{h2^{l-g}+2^l}\big)f\big\|_2^2.
        \end{split}
    \end{equation*}
    Using \eqref{eq01} and \eqref{eq47}, and next the orthonormality of the system $(\widetilde{\chi_y}:y\in\mathbb I^n)$, we get
    \begin{equation}\label{eq28}
        \begin{split}
            \mathcal I_2^2 &\lesssim \sum_{l=0}^\infty \sum_{g=0}^l 2^{\frac 12g} \sum_{\begin{smallmatrix}
                h\in\{1,...,2^g\}\\
                h2^{l-g}+2^l\leqslant \lfloor\frac n2\rfloor
            \end{smallmatrix}} \left\| \sum_{\begin{smallmatrix}
                y\in\mathbb I^n\\
                |y|\leqslant\frac n2
            \end{smallmatrix}}\widehat f(y) \Big(\kappa^{(n)}_{(h-1)2^{l-g}+2^l}(|y|)-\kappa^{(n)}_{h2^{l-g}+2^l}(|y|)\Big) \widetilde{\chi_y} \right\|_2^2\\
            &= \sum_{l=0}^\infty \sum_{g=0}^l 2^{\frac 12g} \sum_{\begin{smallmatrix}
                h\in\{1,...,2^g\}\\
                h2^{l-g}+2^l\leqslant \lfloor\frac n2\rfloor
            \end{smallmatrix}} \sum_{\begin{smallmatrix}
                y\in\mathbb I^n\\
                |y|\leqslant\frac n2
            \end{smallmatrix}} \Big|\widehat f(y) \Big(\kappa^{(n)}_{(h-1)2^{l-g}+2^l}(|y|)-\kappa^{(n)}_{h2^{l-g}+2^l}(|y|)\Big)\Big|^2\\
            &= \sum_{\begin{smallmatrix}
                y\in\mathbb I^n\\
                |y|\leqslant\frac n2
            \end{smallmatrix}} \Big|\widehat f(y)\Big|^2\Psi(|y|),
        \end{split}
    \end{equation}
    where
    \begin{equation*}
        \Psi(x):= \sum_{l=0}^\infty \sum_{g=0}^l 2^{\frac 12g} \sum_{\begin{smallmatrix}
                h\in\{1,...,2^g\}\\
                h2^{l-g}+2^l\leqslant \lfloor\frac n2\rfloor
            \end{smallmatrix}} \Big|\kappa^{(n)}_{(h-1)2^{l-g}+2^l}(x)-\kappa^{(n)}_{h2^{l-g}+2^l}(x)\Big|^2\qquad\text{for } x\in\big\{0,...,\big\lfloor\tfrac n2\big\rfloor\big\}.
    \end{equation*}
    We will prove that uniformly
    \begin{equation}\label{eq29}
        \Psi(x)\lesssim 1\qquad\text{for } x\in\big\{0,...,\big\lfloor\tfrac n2\big\rfloor\big\}.
    \end{equation}
    Then by \eqref{eq28} we will get
    \begin{equation*}
        \mathcal I_2^2\lesssim \sum_{\begin{smallmatrix}
                y\in\mathbb I^n\\
                |y|\leqslant\frac n2
            \end{smallmatrix}} \Big|\widehat f(y)\Big|^2=\|f\|_2^2,
    \end{equation*}
    which is our goal.\\
    Let us take any $x\in\big\{0,...,\big\lfloor\tfrac n2\big\rfloor\big\}$. If $x=0$ then we easily see that $\Psi(x)=0$. So now let us assume that $x>0$.\\
    For every $l\in\{0,1,2,...\},\;\; g\in\{0,...,l\}$, and $h\in\{1,...,2^g\}$ such that $h2^{l-g}+2^l\leqslant\big\lfloor\frac n2\big\rfloor$, by Lemma \ref{thm:szac_dla_krawtch}(a) we get
    \begin{equation*}
        \begin{split}
            \Big|\kappa^{(n)}_{(h-1)2^{l-g}+2^l}(x)-\kappa^{(n)}_{h2^{l-g}+2^l}(x)\Big| &\leqslant \Big|\kappa^{(n)}_{(h-1)2^{l-g}+2^l}(x)-1\Big|+\Big|1-\kappa^{(n)}_{h2^{l-g}+2^l}(x)\Big|\\
            &\lesssim\frac{((h-1)2^{l-g}+2^l)x}n+ \frac{(h2^{l-g}+2^l)x}n\leqslant 4\frac{2^lx}n
        \end{split}
    \end{equation*}
    while by Lemma \ref{thm:szac_dla_krawtch}(b) we have
    \begin{equation*}
        \begin{split}
            \Big|\kappa^{(n)}_{(h-1)2^{l-g}+2^l}(x)-\kappa^{(n)}_{h2^{l-g}+2^l}(x)\Big| &\leqslant \Big|\kappa^{(n)}_{(h-1)2^{l-g}+2^l}(x)\Big|+\Big|\kappa^{(n)}_{h2^{l-g}+2^l}(x)\Big|\\
            &\lesssim\frac n{((h-1)2^{l-g}+2^l)x}+ \frac n{(h2^{l-g}+2^l)x}\leqslant 2\frac n{2^lx}.
        \end{split}
    \end{equation*}
    In summary, the above two inequalities imply
    \begin{equation}\label{eq37}
        \Big|\kappa^{(n)}_{(h-1)2^{l-g}+2^l}(x)-\kappa^{(n)}_{h2^{l-g}+2^l}(x)\Big|\lesssim\min\Big\{\frac{2^lx}n,\frac n{2^lx}\Big\}.
    \end{equation}
    Next, by Lemma \ref{thm:szac_dla_krawtch}(c) we get
    \begin{equation}\label{eq38}
        \begin{split}
            \Big|\kappa^{(n)}_{(h-1)2^{l-g}+2^l}&(x)-\kappa^{(n)}_{h2^{l-g}+2^l}(x)\Big|\leqslant\sum_{k=(h-1)2^{l-g}+2^l+1}^{h2^{l-g}+2^l} \Big|\kappa^{(n)}_{k-1}(x)-\kappa^{(n)}_k(x)\Big|\\
            &\lesssim \sum_{k=(h-1)2^{l-g}+2^l+1}^{h2^{l-g}+2^l} \frac 1k\leqslant \sum_{k=(h-1)2^{l-g}+2^l+1}^{h2^{l-g}+2^l} \frac 1{2^l}=2^{l-g}\frac 1{2^l}=2^{-g}.
        \end{split}
    \end{equation}
    By \eqref{eq37} and \eqref{eq38} we obtain
    \begin{equation*}
        \Big|\kappa^{(n)}_{(h-1)2^{l-g}+2^l}(x)-\kappa^{(n)}_{h2^{l-g}+2^l}(x)\Big|^2\lesssim\min\Big\{\frac{2^lx}n,\frac n{2^lx}\Big\}^{\frac 14}2^{-\frac 74g}.
    \end{equation*}
    Hence for every $l$ and $g$ we get
    \begin{equation*}
        \begin{split}
            2^{\frac 12g} \sum_{\begin{smallmatrix}
                h\in\{1,...,2^g\}\\
                h2^{l-g}+2^l\leqslant \lfloor\frac n2\rfloor
            \end{smallmatrix}} \Big|\kappa^{(n)}_{(h-1)2^{l-g}+2^l}(x)-\kappa^{(n)}_{h2^{l-g}+2^l}(x)\Big|^2&\lesssim 2^{\frac 12g}2^g \min\Big\{\frac{2^lx}n,\frac n{2^lx}\Big\}^{\frac 14}2^{-\frac 74g}\\
            &=2^{-\frac 14g} \min\Big\{\frac{2^lx}n,\frac n{2^lx}\Big\}^{\frac 14}.
        \end{split}
    \end{equation*}
    Hence by \eqref{eq32} we obtain
    \begin{equation*}
        \Psi(x)\lesssim\sum_{l=0}^\infty \sum_{g=0}^l 2^{-\frac 14g} \min\Big\{\frac{2^lx}n,\frac n{2^lx}\Big\}^{\frac 14}\lesssim\sum_{l=0}^\infty \min\Big\{\frac{2^lx}n,\frac n{2^lx}\Big\}^{\frac 14}\lesssim 1.
    \end{equation*}
    Thus \eqref{eq29} is proved.
\end{proof}

\end{document}